\documentclass[12 pt]{amsart}
\usepackage{amssymb,amsfonts,amsthm}
\usepackage{graphicx}
\usepackage{bm}
\newtheorem{thm}{Theorem}[section]
\newtheorem{cor}[thm]{Corollary}
\newtheorem{lem}[thm]{Lemma}
\newtheorem{pro}[thm]{Proposition}

\theoremstyle{definition}
\newtheorem{dfn}[thm]{Definition}
\theoremstyle{remark}
\newtheorem{rem}[thm]{Remark}
\numberwithin{equation}{section}

\begin{document}
\title[Polynomial convergence rate to NESS]{On the polynomial convergence
  rate to nonequilibrium steady states}
\author{Yao Li}
\address{Yao Li: Department of Mathematics and
    Statistics, University of Massachusetts Amherst, Amherst, MA, 01003}
\email{yaoli@math.umass.edu}


\begin{abstract}
We consider a stochastic energy exchange model that models the 1D
microscopic heat conduction in the nonequilibrium setting. In this
paper, we prove the existence and uniqueness of the
nonequilibrium steady state (NESS) and, furthermore, the
polynomial speed of convergence to the NESS. Our result shows that the
asymptotic properties of this model and its deterministic dynamical
system origin are consistent. The proof uses a new technique called
the induced chain method. We partition the state space and work on both
the Markov chain induced by an ``active set'' and the tail of return
time to this ``active set''.
\end{abstract}

\maketitle
\section{Introduction}

As a ubiquitous process, heat conduction has been studied for over two
hundred years. However, from a mathematical point of view, many microscopic aspects of heat conduction in
solids and gas are still unclear. For example, the derivation of macroscopic
heat conduction laws like Fourier's law from microscopic Hamiltonian
dynamics is a well-known challenge in statistical mechanics for the
past over a
century. Over the last several decades, numerous mathematical models
of 1-D microscopic heat conduction have been proposed and
studied. Some of these models have purely deterministic dynamics
\cite{ 
  eckmann2006nonequilibrium, lefevere2010hot,   dolgopyat2011energy, ruelle2012mechanical}, while others are defined by stochastic differential
equations \cite{rey2003nonequilibrium, liverani2012toward, rey2000asymptotic,
  rey2002fluctuations, eckmann1999nonequilibrium, bernardin2005fourier} or Markov jump
processes \cite{grigo2012mixing, li2013existence,
  li2014nonequilibrium, sasada2013spectral}.  These
models give mathematical frameworks for studying nonequilibrium
phenomena including basic properties of nonequilibrium steady states
(NESS), thermal conductivity, local thermodynamic equilibrium (LTE), fluctuation
theorems, and eventually Fourier's law.

This paper focuses on fundamental properties of nonequilibrium
phenomena including the existence and uniqueness of the NESS and,
furthermore, the polynomial-speed convergence to the NESS, for a class of 1-D
microscopic heat conduction models. The model we study is a stochastic
energy exchange model that is inspired by the
KMP model introduced in \cite{kipnis1982heat}, in which a chain of $N$ sites are coupled with
two heat baths. Each site carries a certain amount of energy. An
exponential clock is associated with each pair of adjacent sites. When
the clock rings, these two sites exchange energy in a
``random halves'' fashion. The energy exchange with the bath follows a
similar rule. Different from the model in \cite{kipnis1982heat}, the
rate of an exponential clock here is
energy-dependent. In this paper, the clock rate between two adjacent
sites, called the {\it stochastic energy exchange rate}, depends on
the square root of the minimum of two site energies. We refer Section
\ref{modeldesc} for the precise description of the model.

\begin{figure}[h]
\centerline{\includegraphics[width = 0.8\linewidth]{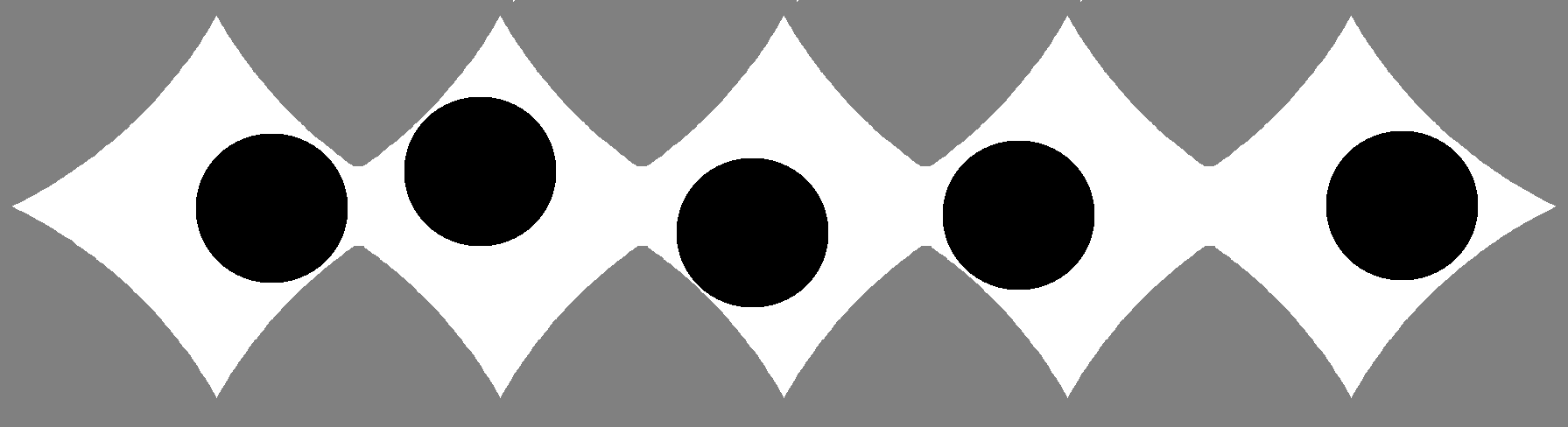}}
\caption{Locally confined particle system\label{fig:1} }
\end{figure}
The motivation of letting the stochastic energy exchange rate depend
on the minimum of site energy comes from the study of a deterministic
dynamical system heat conduction model, called the locally confined particle
system. Introduced in \cite{bunimovich1992ergodic}, the locally confined particle system is a chain of locally confining cells in
$\mathbb{R}^{2}$ like in Figure \ref{fig:1}. An identical rigid
disk-shaped moving particle is contained in each cell. A particle can not pass through the ``bottleneck'' between adjacent cells but can collide
with its neighbors. Therefore, kinetic energy can be exchanged by
particle-particle collisions. Studying such a purely deterministic
dynamical system, especially in the non-equilibrium setting, is very
challenging. Only very limited rigorous results are known. On the other hand,
it is well-known that chaotic billiard systems like the Lorentz gas
have many stochastic properties due to the quick decay of
correlation \cite{chernov1999decay, chernov2000decay,
  chernov2006chaotic, baladi2015exponential}. Hence a natural approach is to only record the kinetic
energy of each particle and approximate this model by the stochastic
energy exchange model described above. Let $E_{i}$ and $E_{i+1}$ be kinetic energies of adjacent
particles. Assume the geometry of the cell allows a particle to be able to
completely ``hide'' from its neighbors, i.e., neighboring particles
can not collide with a particle when it is located at some area of its
cell. Our numerical simulations in \cite{li2015stochastic} show
that when starting from a fixed energy configuration $(E_{i}, E_{i+1})$, the first particle-particle collision time  is well approximated by an
exponential distribution. Heuristically, this is an expected result
because for a sufficiently chaotic dynamical system, the rescaled ``return
time'' and ``hitting time'' to asymptotically small set both converge to
the same exponential distribution \cite{haydn2005hitting}. The
numerical simulation in \cite{li2015stochastic} further shows that the rate of the exponential distribution for the first
collision time can be approximated by $\sim \sqrt{\min \{ E_{i},
  E_{i+1} \}}$ when one of $E_{i}$ and $E_{i+1}$ is sufficiently small. The heuristic reason of this rate is that
when one particle is sufficiently slow and out the reach of its
neighbors, the next particle-particle collision time should be primarily determined by
the kinetic energy of the slow particle.   

We remark that at a certain time rescaling limit, the locally confined 
particle system may have a different stochastic energy exchange rate. In
a non-rigorous study of the locally confined particle system
\cite{gaspard2008heat, gaspard2008heat2}, a rate function $\sim \sqrt{E_{i} +
E_{i+ 1}}$ is obtained at a certain rare interaction limit and time rescaling
limit. Assuming this rate of interaction, the mixing rate is 
known to be exponential \cite{grigo2012mixing,
  li2013existence, sasada2013spectral}. Without any time rescaling limit, it is a
simple mathematical fact that the speed of mixing in the locally
confined particle system can not be faster than $t^{-2}$, provided
particles can ``hide'' from their neighbors. (See lemma 3.1 of \cite{li2015stochastic}.) This $\sim t^{-2}$ speed
of mixing is also one of the main result of this paper. Therefore, the approximate interaction rate
$\sim \sqrt{\min \{ E_{i}, E_{i+1} \}}$ computed in \cite{li2015stochastic} is consistent with
the asymptotical dynamics of the locally confined particle system at
its original time scale.

In this paper, among other results, we proved that when the stochastic rate of energy
exchange between two sites are $\sim \sqrt{\min \{ E_{i},
  E_{i+1}\} }$, the Markov chain generated by our model has $\sim
t^{-2}$ rate of mixing and $\sim t^{-2}$ rate of contraction. These results completely
match our analytical and numerical results about the locally confined
particle system in \cite{li2015stochastic}. As shown in the proof later in this
paper, the main source of the slow-speed
mixing comes from the rate $\sim \sqrt{\min \{ E_{i}, E_{i+1} \} } $. When one
site acquires a very low amount of energy from an energy exchange, the
rates of two corresponding clocks become very low and can not be ``rescued'' by other
``faster clocks''. Hence the next energy exchange at this site will not happen
within a long time period, which obviously slows down the speed of mixing and convergence. We
remark that this is also consistent with the mechanism of slow-speed mixing
phenomenon of the locally confined particle system.

In addition to the ergodicity, the quantitative property of the
NESS is also of great interest. Our result shows the absolute
continuity of the NESS with respect to the Lebesgue measure. In
addition, we obtain the tail of the first passage time to a certain
uniform reference set. This helps us to show that the tail of the marginal
distribution of NESS with respect to each site is $\geq E^{-1/2}$ when
$E \ll 1$. Since the explicit formulation of the NESS usually can not be given, a detailed
study on the properties of the NESS will rely heavily on numerical
simulations. We will write a separate paper to numerically study the NESS,
the long-range correlation, and the thermal conductivity of the
generalized KMP model studied in this paper. 

Despite the straightforward heuristic argument, a rigorous proof of
the slow-speed mixing of a Markov process is known to be
difficult. To the best of our knowledge, our result is the first polynomial
convergence result in non-equilibrium settings. We prove that the
Markov process generated by the generalized KMP model has a mixing
rate $\sim t^{-2}$ and a convergence rate $\sim t^{-1}$ to the NESS. The closest related
results we know are the slower-than-exponential convergence to the NESS in
\cite{yarmola2013sub, yarmola2014sub, cuneo2016non, cuneo2015non} and the polynomial convergence to the equilibrium in
\cite{li2016polynomial}. In addition to the upper bound of
convergence, we also showed that the speed of convergence to NESS has
a lower bound $t^{-1-\gamma}$ for any $\gamma > 0$. This further
confirms the polynomial ergodicity.

The method of proving the polynomial-speed convergence to the
steady-state is called the induced chain method. Since the source of
slow convergence is the low-energy site, we partition the phase into
two parts: the ``active'' set and the ``inactive'' set, where the
``active'' set means all site energies are above a certain
threshold. Then we work on the Markov chain induced by the ``active''
set. Different from the model in \cite{li2016polynomial}, where the ``active''
set satisfies a Doeblin-type condition, in this paper we need some
extra work to show the stochastic stability of the induced chain. The
induced chain method consists of three steps. We
first show that the induced chain admits a uniform reference set, on which trajectories can be coupled with
strictly positive probability. Then we control the first passage time to this uniform
reference set for the induced chain, this is done by constructing a Lyapunov function as we have
done in \cite{li2016polynomial}. Last, we show that the time duration of one step of the
induced chain has a polynomial tail by a technical construction of
Lyapunov functions. A global Lyapunov function is obtained from a
``tower construction'' of local Lyapunov functions with respect to
nearest neighbor interactions. The three steps above imply that the first passage time of
the full system to the uniform reference set has a polynomial
tail. Then we can apply results from discrete renewal theory and prove the
polynomial-speed convergence and mixing. 

We remark that a common way of proving polynomial-speed convergence
is to construct a Lyapunov function \cite{li2016polynomial, hairer2010convergence}. However, in this
model such a construction is too complicated to be practical. One
needs to consider both the lack of
tightness and the possible inactive clocks in the construction of a
Lyapunov function. By using the induced chain method, we can treat
these two problems separately, and eventually give the tail of the
first passage time to a uniform reference set. This method is useful
in proving the polynomial (or sub-exponential) convergence of other
models. In addition, we believe the induced chain technique can be
extended into a hierarchy of finitely many induced chains and be
applied to a wider range of problems. 

The paper is organized in the following way: Section \ref{modelresult}
introduces the model and states the main result. The main strategy of proof,
i.e., the induced chain method, is introduced in Section
\ref{approach}. Estimations for the time duration of one step of the
induced chain are given in Section \ref{low}. The first entry time to
the uniform reference set of the induced chain is done in Section
\ref{high}. Finally, we complete the whole proof in Section \ref{pf}.

\section{Model and Result}\label{modelresult}
\subsection{Stochastic approximation of deterministic dynamics}

We start with a short review of the locally confined particle system
and its stochastic approximation. Consider a chain of cells in
$\mathbb{R}^{2}$ that are formed by finitely many piecewise $C^{3}$
curves. A rigid disk-shaped {\it moving particle} is confined in each
cell, as shown in Figure \ref{fig:1}. Two adjacent cells are connected by a ``bottleneck'' opening
such that particles can not pass the opening but can collide with each
other. A particle moves freely until it collides with the cell boundary
or its neighbor particles. In addition, we assume each cell forms a
strongly chaotic billiard table. In the absence of other
particles, the billiard map of one particle is exponentially
mixing. We refer to \cite{bunimovich1992ergodic} for the ergodicity of the locally
confined particle system under suitable conditions and \cite{chernov2006chaotic}
for major results of dynamic billiards. 

In the locally confined particle system, a particle has a quick decay of
correlation due to frequent collisions with the cell
boundary. Therefore, it is natural to simplify the model by assuming
that the geometry within a cell is forgotten by the particle. More
precisely, we only record the kinetic energy carried by a particle and
assume that the time to the next energy exchange is exponentially
distributed. The rate of this exponential distribution is called
the {\it stochastic energy exchange rate}, denoted by $R(E_{i},
E_{i+1})$, where $E_{i}$ and $E_{i+1}$ are kinetic energies of particles. If in addition, we assign a
suitable rule for the energy redistribution in a
particle-particle collision, a Markov jump process is obtained. 

In \cite{li2015stochastic}, we numerically showed that the time to the next
particle-particle collision always has an exponential tail that depends on the
energy configuration of particles. This further supports the idea of
approximating the locally confined particle system by a Markov jump
process. The rate $R(E_{i}, E_{i+1})$ can be
computed numerically as the slope of the exponential tail of the waiting time to 
the next particle-particle collision. Our numerical result in
\cite{li2015stochastic} showed that $R(E_{i},
E_{i+1}) \sim \sqrt{ \min \{ E_{i}, E_{i+1} \}}$ when at least one of $E_{i}$
or $E_{i+1}$ is sufficiently small. 

As stated in the introduction, the slow convergence phenomenon comes from the slow clock rate
when one of $E_{i}$ or $E_{i+1}$ is very small. To preserve this
qualitative property of the model, it is sufficient to
let the stochastic energy exchange rate be $\sqrt{ \min \{ E_{i},
  E_{i+1} \}}$ for all small $\min \{ E_{i}, E_{i+1} \}$. Hence we
assume $R(E_{i}, E_{i+1}) = \min \{K, \sqrt{\min(E_{i}, E_{i+1})} \} $
in this paper for the sake of simplicity. We idealize the energy exchange in a particle-particle
collision by choosing it as a ``random halves'' energy
redistribution. Our numerical simulation shows that this is a
reasonable choice, as the amount of redistributed energy has a strictly
positive probability density function. In addition, the system is coupled with two heat baths and
we prescribe a similar rule for the energy exchange with the heat
bath. This gives rise to
the stochastic energy exchange model as will be described in the next subsection.

\subsection{Description of the stochastic energy exchange model}\label{modeldesc}
Now we give a precise description of the stochastic energy
exchange model. Consider a chain of $N$ lattice sites connected to two heat
baths at the ends. The energy at each site is denoted by $E_{1},
\cdots, E_{N}$, respectively. The temperature of heat baths are
$T_{L}$ and $T_{R}$. An exponential clock is
associated with each pair of adjacent sites. The rate of the clock
depends on the energy at both sites, denoted by $R(E_{i},
E_{i+1})$. When the clock rings, the energy at two sites are pooled
together and redistributed randomly as
\begin{equation}
\label{exchange}
  (E_{i}', E_{i+1}') = ( p(E_{i} + E_{i+1}), (1-p)(E_{i} + E_{i+1})) \,,
\end{equation}
where $p$ is a uniform random variable distributed on $(0, 1)$ that is
independent of everything else. In
addition, an exponential clock is associated with the first (resp. last) site and
the left (resp. right) heat bath, whose rate is $R( T_{L}, E_{1})$
(resp. $R(E_{N}, T_{R})$) for the same rate function used above. When the clock rings, the energy at the
first (resp. last) site exchanges energy with an exponential random
variable :
$$
  E_{1}' = p( E_{1} + X_{L}) \quad \mbox{( resp. } E_{N}' = p(E_{N} +
  X_{R}) )\,,
$$
where $p$ is a uniform random variable on $(0, 1)$ that is independent
of everything else, $X_{L}$
(resp. $X_{R}$) is an exponential random variable with mean $T_{L}$
(resp. $T_{R}$). All exponential clocks are assumed to be mutually
independent. We remark that the uniform random variable $p$ is chosen
to simplify the proof. Our method works for other choices of $p$ with uniformly positive and bounded density on $(0, 1)$. 

The rate $R(E_{i}, E_{i+1})$, called the {\it stochastic energy
  exchange rate} between sites $i$ and $i+1$, has the following form:
$$
  R(E_{i}, E_{i+1}) = \min \{K, \sqrt{\min(E_{i}, E_{i+1})} \} \,,
$$
where $K \gg 1$ is a sufficiently large constant. As explained above, the
stochastic energy exchange rate is assumed to be the square
root of the minimum of site energies when either of the site energies is sufficiently small. The maximum of
stochastic energy exchange rate is set as $K < \infty$ for technical
reasons. Without such an assumption, the Lyapunov function is not in
the domain of the infinitesimal generator of the Markov chain, which imposes certain
technical complexity \cite{meyn1993stability}. As the aim of this paper is to show that the
property of the rate function at low energy leads to polynomial rate of convergence
to NESS, we choose to cap the energy exchange rate by $K$. $K$ is
assumed to be sufficiently large so that it will not significantly
affect the dynamics at any ``normal configuration''. In particular, we
assume $K \gg T_{L}, T_{R}$.

It is easy to see from the description that this model generates a
Markov jump process $\mathbf{E}_{t} = (E_{1}(t), \cdots, E_{N}(t))$ on
$\mathbb{R}^{N}_{+}$. For any measurable function $f$, the infinitesimal generator of $\mathbf{E}_{t}$ is  
\begin{eqnarray*}
&&\mathcal{L} f(E_{1}, \cdots, E_{N}) \\
& =  & \sum_{i = 1}^{N - 1}
R(E_{i}, E_{i+1})\int_{0}^{1} \left \{f(E_{1}, \cdots, E_{i-1}, p(E_{i} + E_{i+1}), (1-p)(E_{i}
+ E_{i+1}), E_{i+2}, \cdots, E_{N}) \right .\\
&&\left . - f(E_{1}, \cdots, E_{N}) \right \}
\mathrm{d}p \\
&&+ R(T_{L}, E_{1})\int_{0}^{\infty} \int_{0}^{1} \left \{ \frac{1}{T_{L}}
e^{-s/T_{L}}f(p(E_{1} + s), E_{2}, \cdots, E_{N})  - f(E_{1}, \cdots,
   E_{N})\right \}\mathrm{d}p
\mathrm{d}s   \\
&&+ R(E_{N}, T_{R})\int_{0}^{\infty} \int_{0}^{1}\left \{ \frac{1}{T_{R}}
e^{-s/T_{R}}f(E_{1}, \cdots, E_{N-1}, p(E_{N} + s)) - f( E_{1},
\cdots, E_{N}) \right \} \mathrm{d}p
\mathrm{d}s 
\end{eqnarray*}
We denote the transition kernel of $\mathbf{E}_{t}$ by $P^{t}(\mathbf{E},
\cdot)$, where $\mathbf{E} \in \mathbb{R}^{N}_{+}$. The left and right
operator generated by $P^{t}$ are
$$
  (P^{t} \zeta)(\mathbf{E}) = \int_{\mathbb{R}^{N}_{+}} P^{t}(
  \mathbf{E}, \mathrm{d}\mathbf{x}) \zeta (\mathbf{x})
$$
for a measurable function $\zeta( \mathbf{E})$ on $\mathbb{R}^{N}_{+}$,
and
$$
  (\mu P^{t})(B) = \int_{\mathbb{R}^{N}_{+}} P^{t}( \mathbf{E}, B)
  \mu( \mathrm{d} \mathbf{E})
$$
for a probability measure $\mu$ on $\mathbb{R}^{N}_{+}$. We also use
notations $\mathbb{P}_{\mathbf{E}}$ and $\mathbb{E}_{\mathbf{E}}$ for
conditional probability and conditional expectation with respect to
the initial condition $\mathbf{E}_{0} = \mathbf{E}$. 

In this paper we will use energy exchange events frequently. The event that the exponential clock between $E_{i-1}$ and
$E_{i}$ rings at a time $t$ (resp. at a stopping time $\tau$) is denoted
by $\mathcal{C}_{i}(t)$ (resp. $\mathcal{C}_{i}(\tau)$) for all $i =
1, \cdots,  N + 1$.  In addition, for the
sake of simplicity we denote $E_{0} = T_{L}$ and $E_{N+1} =
T_{R}$.


\subsection{Main Result}
To state our main result precisely, the following functions and measure classes are
necessary. Let
$$
  W( \mathbf{E}) = \sum_{ i = 1}^{N} E_{i} \,.
$$
For any $0 < \eta \ll 1$, let
$$
  V( \mathbf{E}) = V_{\eta}( \mathbf{E}) = \sum_{m = 1}^{N} \sum_{i = 1}^{N-m + 1} ( \sum_{j
    = 0}^{m - 1} E_{i + j})^{a_{m} \eta - 1} \,,
$$
where $a_{m } = 1 - (2^{m - 1} - 1)/(2^{N} - 1)$ for $m = 1, \cdots, N$. Note that $0 < a_{m} < 1$ for each $m$, hence all powers $a_{m}
\eta - 1$ are negative. Let $\mathcal{M}_{\eta}$ be the collection of probability measure $\mu$ on
$\mathbb{R}^{N}_{+}$ such that
$$
  \int_{\mathbb{R}^{N}_{+}} (W( \mathbf{E}) +V( \mathbf{E}) )\mu(
  \mathrm{d} \mathbf{E}) < \infty. 
$$

$\mathcal{M}_{\eta}$ covers a large class of probability measures. For
example, if $\mathbf{X}$ is a random energy configuration that has
finite expectation and its density at the neighborhood of the ``boundary'' $\{ \mathbf{E}
= (E_{1}, \cdots, E_{N}) \,|\,  E_{i}  = 0 \mbox{ for some } i = 1
\sim N \}$ is uniformly bounded, then the probability measure induced by $\mathbf{X}$ belongs
to $\mathcal{M}_{\eta}$ for any sufficiently small $\eta > 0$. 

\medskip

We have the following results regarding the stochastic stability of
$\mathbf{E}_{t}$. 

{\bf Theorem 1 (Polynomial contraction of the Markov operator). } {\it For any
$\gamma > 0$, there exists $\eta > 0$ such that 
$$
  \lim_{t\rightarrow \infty} t^{2 - \gamma} \| \mu P^{t} - \nu P^{t}
  \|_{TV} = 0 
$$
for any $\mu$, $\nu \in \mathcal{M}_{\eta}$, where $\| \cdot \|_{TV}$ is the total variation norm.
}

{\bf Theorem 2 (Properties of the invariant
  measure).} {\it There exists a unique invariant measure $\pi$ that is
absolutely continuous with respect to the Lebesgue measure. In
addition, for any $\gamma > 0$ there exists $\eta > 0$ such that
$$
  \lim_{t\rightarrow \infty} t^{1 - \gamma} \| \mu P^{t} - \pi \|_{TV} = 0
$$
for any $\mu \in \mathcal{M}_{\eta}$.
}

Note that $\pi$ may not be in $\mathcal{M}_{\eta}$, which leads to a
different rate in {\bf Theorem 2}. 

A corollary of polynomial contraction of Markov operator is the rate
of correlation decay. 

{\bf Theorem 3 (Polynomial correlation decay). } {\it Let functions
  $\xi$ and $\zeta $ be in $L^{\infty}(
\mathbb{R}^{N}_{+})$. For any $\gamma
  > 0$, there exists $\eta > 0$ such that for any $\mu \in \mathcal{M}_{\eta}$
\begin{eqnarray*}
  &&\left|\int_{\mathbb{R}^{N}_{+} } (P^{t} \zeta)(
  \mathbf{E}) \xi( \mathbf{E}) \mu( \mathrm{d}\mathbf{E}) - \int_{\mathbb{R}^{N}_{+}} (P^{t}\zeta)( \mathbf{E}) \mu(
  \mathrm{d}\mathbf{E}) \int_{\mathbb{R}^{N}_{+}} \xi( \mathbf{E}) \mu(
  \mathrm{d} \mathbf{E}) \right| \ \\
&\leq& O(1) \cdot \|\xi\|_{L^{\infty}} \
\|\zeta\|_{L^{\infty}} \ \left(\frac{1}{t^{2-\gamma}} \right) 
  \end{eqnarray*}
as $t \to \infty$, where the $O(1)$ term depends on $\gamma, N$, and $\mu$. 
}

Finally, the following proposition gives the lower bound of
convergence speed.

{\bf Proposition 4 (Lower bound of convergence)}
{\it
There exists a probability measure $\nu$ satisfying $\mathrm{d}\nu \ll
\mathrm{d}\pi$ such that
$$
  \| \nu P^{t} - \pi \|_{TV} \geq c (t + 1)^{-1 - \gamma }
$$
for any sufficiently small $\gamma > 0$. 
}

\section{Approach towards polynomial ergodicity}\label{approach}
The purpose of this section is to introduce a general approach, called the
induced chain method, towards the polynomial ergodicity of a Markov
process. We introduce this method under the generic setting, as it can
be potentially applied to other models. Our aim is to make this
section self-contained. When citing results from references, we will
explain how statements of those theorems are rephrased.  

Throughout this
section, we let $\Psi_{n}$ be a discrete-time Markov chain on a
measurable space $(X, \mathcal{B})$. The transition kernel of
$\Psi_{n}$ is $\mathcal{P}(x, \cdot)$. 

For $A \in \mathcal{B}$, we let $\tau_{A}$ be the first passage time
to $A$:
$$
  \tau_{A} = \inf \{n > 0 \, | \, \Psi_{n} \in A \} \,.
$$
A set $A \in \mathcal{B}$ is said to be {\it accessible} if
$\mathbb{P}_{x}[ \tau_{A} < \infty]  = 1$ for every $x \in X$. 

We say a Markov chain is {\it
  irreducible} with respect to a measure $\phi$ on $\mathcal{B}$ if
for any $ A \in \mathcal{B}$ with $\phi(A) > 0$ and any $x \in X$,
there exists $n$ such that $\mathcal{P}^{n}(x,A) > 0$. In other word, every set
with positive $\phi$-measure is accessible. We refer readers to Chapter 4 of
\cite{meyn2009markov} for the well-definedness of irreducibility. In
fact, if $\Psi_{n}$ is irreducible with respect to some measure
$\phi$, then there exists a ``maximal irreducible measure'' $\psi$
that is unique up to equivalence class.

\subsection{Splitting, Coupling, and moments of return times}

By the polynomial ergodicity of $\Psi_{n}$, we mean the polynomial rate of
contraction of the Markov operator $\mathcal{P}$, the polynomial rate
of convergence to the invariant measure of $\Psi_{n}$, and the polynomial rate of
correlation decay of $\Psi_{n}$. To prove the polynomial ergodicity,
the bound on return times to a certain {\it uniform reference set} $\mathfrak{C}
\in \mathcal{B}$ is crucial. 

\begin{dfn}
A set $\mathfrak{C} \in \mathcal{B}$ is said to be a uniform reference
set if it satisfies
$$
  \inf_{x \in \mathfrak{C}} \mathcal{P}(x, \cdot) \geq \delta \theta(\cdot) \,,
$$
where $\theta$ is a probability measure on $(X, \mathcal{B})$ and
$\delta$ is a strictly positive real number.
\end{dfn}

A uniform reference set is a special case of {\it small set} or {\it
  petite set} defined in \cite{meyn2009markov}.

We call $\mathfrak{C}$ a uniform reference set because processes starting
from $\mathfrak{C}$ have some uniform ``common future''. If such a uniform
reference set $\mathfrak{C}$ exists, $\Psi_{n}$ can be converted to a
new process $\tilde{\Psi}_{n}$ on a modified state space $\tilde{X} =
X \cup \mathfrak{C}_{1}$, where $\mathfrak{C}_{1}$ is an identical
copy of $\mathfrak{C}_{0}:= \mathfrak{C}$. $\mathcal{B}$ can be extended to
$\mathcal{\tilde{B}}$ on $\tilde{X}$ accordingly. Then we can split a
probability measure $\mu$ on $(X, \mathcal{B})$ to a probability  measure $\mu^{*}$ on $(\tilde{X},
\mathcal{\tilde{B}})$:
$$
  \left \{
\begin{array}{ll}
 & \mu^*|_X = (1 - \delta)\ \mu|_{\mathfrak{C}_0} + \mu|_{X \setminus \mathfrak{C}_0}\\
&  \mu^*|_{\mathfrak{C}_1} = \delta \ \mu|_{\mathfrak{C}_0}\ , \quad \mathfrak{C}_0 \cong \mathfrak{C}_1  \mbox{ via the natural
identification }.
\end{array}
\right .
$$

Then we can ``lift'' $\Psi_{n}$ to a new Markov process
$\tilde{\Psi}_{n}$ on $\tilde{X}$ with a transition kernel
$\tilde{\mathcal{P}}(x, \cdot)$:

$$
  \left \{ 
\begin{array}{cl}
 \mathcal{\tilde P}(x, \cdot) = (\mathcal{P}(x, \cdot))^* &  x \in X \setminus \mathfrak{C}_0\\
 \mathcal{\tilde P}(x, \cdot) = [(\mathcal{P}(x, \cdot))^{*}
 - \delta {\theta^{*}(\cdot)}]/(1 - \delta) 
 & x \in \mathfrak{C}_0 \\
\mathcal{\tilde P}( x, \cdot) = {\theta^{*}}(\cdot) & x \in \mathfrak{C}_1
\end{array}
\right .
$$

It is straightforward to check that $\tilde{\Psi}_{n}$ has an atom
$\alpha := \mathfrak{C}_{1}$, i.e., $\tilde{\mathcal{P}}(x, \cdot)$ is the same for all $x \in
\alpha$. In addition, if the initial distribution of
$\tilde{\Psi}_{n}$ is splited from some $\mu$ on $(X, \mathcal{B})$ as
described above, $\tilde{\Psi}_{n}$ projects to
$\Psi_{n}$ through the natural projection from $\tilde{X}$ to
$X$. This transformation is called the Nummelin splitting \cite{nummelin1978splitting}. We
refer to \cite{nummelin1978splitting, meyn2009markov} for the details.

\medskip

If $\Psi_{n}$ admits an accessible uniform reference set, many
results about the stochastic stability of $\Psi_{n}$ can be implied by
estimates about $\tau_{\mathfrak{C}}$. To state the results, we need the
aperiodicity of $\Psi_{n}$.

Let $\mathfrak{C}$ be a uniform reference set. Define $E_{\mathfrak{C}} \subset
\mathbb{Z}^{+}$ be the set of intergers such that $M \in E_{\mathfrak{C}}$ if and
only if
$$
  \inf_{x \in \mathfrak{C}} \mathcal{P}^{M}(x, \cdot) \geq \delta
  \theta(\cdot) \quad, \quad \theta(\mathfrak{C}) > 0
$$
for a probability measure $\theta$ and a strictly positive number
$\delta$. The greatest common divisor of $E_{\mathfrak{C}}$ is called the {\it
  period} of $\Psi_{n}$ with respect to $\mathfrak{C}$. 

\begin{dfn}
An irreducible Markov process $\Psi_{n}$ is said to be {\it aperiodic}
if the period of $\Psi_{n}$ with respect to any uniform reference set
$\mathfrak{C}$ is $1$.
\end{dfn}

\begin{dfn}
An irreducible Markov process $\Psi_{n}$ is said to be {\it strongly aperiodic} if $\Psi_{n}$ admits
a uniform reference set $\mathfrak{C}$ such that $\theta( \mathfrak{C}) > 0$. 
\end{dfn}

If $\Psi_{n}$ is strongly aperiodic, $\Psi_{n}$ must be aperiodic such that
no cyclic decomposition is possible. We refer to Theorem 5.4.4 of
\cite{meyn2009markov} for the precise result about the cyclic
decomposition for Markov processes on measurable state spaces. 

\begin{dfn}
A probability measure $\pi$ is said to be {\it invariant} if
$$
  \pi(A) = \int_{X} \pi( \mathrm{d}x ) \mathcal{P}(x, A)
$$
for any $A \in \mathcal{B}$. 
\end{dfn}

\begin{thm}
\label{exist}
Let $\Psi_{n}$ be an irreducible Markov chain on $(X, \mathcal{B})$ with transition
kernel $\mathcal{P}$. If $\mathfrak{C} \in \mathcal{B}$ is an
accessible uniform reference set such that 
$$
  \sup_{x \in \mathfrak{C}} \mathbb{E}_{x}[ \tau_{\mathfrak{C}}] <
  \infty \,,
$$
then there exists an invariant probability measure $\pi$. 
\end{thm}
\begin{proof}
Define 
$$
  L(x, A) = \mathbb{P}_{x}[ \tau_{A} < \infty] \,.
$$
Then Obviously $L(x, \mathfrak{C})  = 1$ for every $x \in
\mathfrak{C}$. By Theorem 8.3.6 of \cite{meyn2009markov}, $\Psi_{n}$
is recurrent. The theorem then follows from Theorem 10.0.1 of
\cite{meyn2009markov}. 
\end{proof}

The polynomial ergodicity of $\Psi_{n}$
follows from the finiteness of moments of $\tau_{\mathfrak{C}}$.   

\begin{thm}[Theorem 2.6 and 2.7 in \cite{nummelin1983rate}]
\label{thm31}
Let $\Psi_{n}$ be an aperiodic Markov chain on $(X, \mathcal{B})$ with transition
kernel $\mathcal{P}$. Let $\pi$ be an invariant probability measure of
$\Psi_{n}$. If $\mathfrak{C} \in \mathcal{B}$ is an accessible uniform
reference set and 
$$
  \sup_{x\in \mathfrak{C}} \mathbb{E}_{x}[ \tau_{\mathfrak{C}}^{\beta}] < \infty
$$
for some $\beta > 0$, then for any probability measures $\mu, \nu$ on
$X$ that satisfy
$$
  \mathbb{E}_{\mu}[ \tau_{\mathfrak{C}}^{\beta}] < \infty \quad \mathrm{and}
  \quad  \mathbb{E}_{\nu}[ \tau_{\mathfrak{C}}^{\beta}] < \infty \,,
$$ 
we have
$$
  \lim_{n \rightarrow \infty} n^{\beta}\| \mu \mathcal{P}^{n} - \nu
  \mathcal{P}^{n} \|_{TV} = 0 \,.
$$
In addition, if
$\beta > 1$, then
$$
  \lim_{n \rightarrow \infty} n^{\beta - 1}\| \mu \mathcal{P}^{n} - \pi \|_{TV} = 0 
$$
for any $\mu$ which satisfies
$$
  \mathbb{E}_{\mu}[ \tau_{\mathfrak{C}}^{\beta}] < \infty \,.
$$
\end{thm}
\begin{rem}
We rephrased statements of Theorem 2.6 and 2.7 in
\cite{nummelin1983rate} to make the notations consistent. Functions
$\psi(n)$ and $\psi_{0}(n)$ defined in \cite{nummelin1983rate} correspond to $n^{\beta - 1}$ and
$n^{\beta}$. Then by Theorem 2.7 (i) of \cite{nummelin1983rate}, 
$$
  \sup_{x\in \mathfrak{C}} \mathbb{E}_{x}[ \tau_{\mathfrak{C}}^{\beta}] < \infty
$$
implies that $\Psi_{n}$ is ergodic of order $\psi$. By Theorem 2.6, we
have
$$
  \mathbb{E}_{\pi}[ \tau_{\mathfrak{C}}^{\beta - 1}] < \infty \,.
$$
The rest of the results follows from Theorem 2.7 (iii) of \cite{nummelin1983rate}.
\end{rem}

Below we give a short self-contained proof for Theorem \ref{thm31} based on the discrete renewal
theory in \cite{lindvall2002lectures}. We refer to
\cite{nummelin1983rate, li2016polynomial} for the full details, and
\cite{hairer2010convergence} for a modern treatment of continuous time Feller
processes. 

\begin{proof}
We first apply the Nummelin splitting to $\Psi_{n}$ to
obtain $\tilde{\Psi}_{n}$ on $\tilde{X}$. $\tilde{\Psi}_{n}$ possesses
an accessible atom $\alpha$. 

Let $\tilde{\Psi}^{1}_{n}$ and $\tilde{\Psi}^{2}_{n}$ be two
independent copies of $\tilde{\Psi}_{n}$ starting from $\mu^{*}$ and
$\nu^{*}$, respectively. Let $Y^{1}_{0}, Y^{1}_{1}, Y^{1}_{2}, \cdots$
and $Y^{2}_{0}, Y^{2}_{1}, Y^{2}_{2}, \cdots$ be the passage times to
$\alpha$ for the two independent processes, respectively. Since $\alpha$
is an atom, $\{Y^{1}_{i}\}_{i = 1}^{\infty}$ and $\{Y^{2}_{i}\}_{i =
  1}^{\infty}$ are i.i.d random variables. Therefore 
$$
    S^{1}_{n} := \sum_{i = 0}^{n} Y^{1}_{i}, \mathrm{and} \quad  S^{2}_{n} :=
  \sum_{i = 0}^{n} Y^{2}_{i}
$$
form two delayed renewal processes. $Y^{1}_{0}$ and $Y^{2}_{0}$ are called the
{\it delay distributions} of the renewal processes. Let $T$ be the first simultaneous renewal
time
$$
  T = \inf \{ m \geq 0 \,|\, S^{1}_{k_{1}} = S^{2}_{k_{2}} = m \mbox{ for
    some } k_{1}, k_{2 } \}\,.
$$
Since $\Psi_{n}$ (and $\tilde{\Psi}_{n}$) is aperiodic, the renewal
processes $S^{1}_{n}$ and $S^{2}_{n}$ are aperiodic.

Then it is easy to see that after $T$, $\tilde{\Psi}^{1}_{n}$ and
$\tilde{\Psi}^{2}_{n}$ become indistinguishable. $T$ is called the
{\it coupling time} of $\tilde{\Psi}^{1}_{n}$ and
$\tilde{\Psi}^{2}_{n}$. It is well-known that 
$$
 \| \mu \mathcal{P}^{n} - \nu
\mathcal{P}^{n} \|_{TV}  \leq \| \mu^{*} \tilde{\mathcal{P}}^{n} - \nu^{*}
\tilde{\mathcal{P}}^{n} \|_{TV}  \leq 2\mathbb{P}_{\mu^{*}, \nu^{*}}[T > n] \,.
$$

Therefore it is sufficient to show the polynomial tail of
$\mathbb{P}[T > n]$. Instead of the polynomial tail, we first prove
the finiteness of moments of $T$ by using the delayed renewal processes
constructed above and the following two lemmas.

\medskip

\begin{lem}
\label{nummelin}
Let $\Psi_{n}$ be an aperiodic Markov chain on $(X, \mathcal{B})$ with transition
kernel $\mathcal{P}$. If $\mathfrak{C} \in \mathcal{B}$ is an accessible uniform
reference set and 
$$
  \sup_{x\in \mathfrak{C}} \mathbb{E}_{x}[ \tau_{\mathfrak{C}}^{\beta}] < \infty
$$
for some $\beta > 0$, then for any probability measure $\mu$ such
that
$$
  \mathbb{E}_{\mu}[ \tau_{\mathfrak{C}}^{\beta}] < \infty \,,
$$
there exists a constant $C < \infty$, such
that
$$
  \mathbb{E}_{\mu^{*}}[ \tau_{\alpha}^{\beta}] \leq C
  \mathbb{E}_{\mu}[ \tau_{\mathfrak{C}}^{\beta}] < \infty \,.
$$
\end{lem}
\begin{proof}
Apply Nummelin splitting to $\Psi_{n}$ with respect to
$\mathfrak{C}$. Define the stopping time $\tau =
\tau_{\mathfrak{C}_{0} \cup \alpha}$ for $\tilde{\Psi}_{n}$. Let
$\{\tau^{n}\}_{n = 1}^{\infty}$ be the sequence of iterates of $\tau$,
i.e., 
$$
  \tau^{0} = 0, \quad \tau^{1} = \tau, \quad  \tau^{n+1} = \tau^{n} +
  \tau \circ \Theta^{\tau^{n}} \,,
$$
where
$\Theta$ is the usual shift operator. Further let $Z_{n}$ be a
sequence of $\{0, 1\}$ random variables such that $Z_{n} = 1$ if and
only if $\tilde{\Psi}_{\tau^{n}} \in \alpha$. According to the
definition of $\tilde{\Psi}_{n}$, the probability of $\tilde{\Psi}_{\tau^{n}}
= \alpha$ is at least $\delta$ whenever the split chain $\tilde{\Psi}_{n}$ jumps to
$\mathfrak{C}_{0} \cup \alpha$ at the step $\tau^{n}$. Hence $Z_{n}$
is $\mathcal{F}_{n} = \sigma \{\tilde{\Psi}_{0}, \cdots, \tilde{\Psi}_{\tau^{n}}\}$ measurable and 
$$
  \mathbb{P}[Z_{n} = 1 \,|\, \mathcal{F}_{n-1} ] \geq \delta > 0. 
$$
Let $\zeta = \inf \{  n > 0 \,|\, Z_{n} = 1 \} $. Then $\tau_{\alpha} =
\tau^{\zeta}$. The lemma then follows from Lemma 3.1 (iii) of
\cite{nummelin1983rate}. From the proof of Lemma 3.1 of
\cite{nummelin1983rate} we can see that there exists a constant $C < \infty$, such
that
$$
  \mathbb{E}_{\mu^{*}}[ \tau_{\alpha}^{\beta}] \leq C
  \mathbb{E}_{\mu}[ \tau_{\mathfrak{C}}^{\beta}] < \infty \,.
$$
\end{proof}

\medskip

\begin{lem}
\label{lindvall}
Let $S^{1}_{n}$ and $S^{2}_{n}$ be the delayed renewal processes as above. If
there exists $\beta > 1$ such that $\mathbb{E}[(Y_{0}^{1})^{\beta}],
\mathbb{E}[(Y_{0}^{2})^{\beta}]$, and $\mathbb{E}[(Y_{1}^{1})^{\beta}]$
are all finite, then there exists a constant $C< \infty$ depending on $\mathbb{E}[(Y_{1}^{1})^{\beta}]$, such that
$$
  \mathbb{E}[T^{\beta}] < C ( \mathbb{E}[(Y_{0}^{1})^{\beta}] +
  \mathbb{E}[(Y_{0}^{2})^{\beta}] ) < \infty \,.
$$
In addition, there exists a delay distribution $Y^{c}_{0}$ with
$\mathbb{E}[(Y^{c}_{0})^{\beta - 1}] < \infty $ such that
$$
  S^{c}_{n} := Y^{c}_{0} + \sum_{i = 1}^{n} Y^{c}_{i} , \quad n \geq 1 
$$
is a stationary renewal process, where $\{Y^{c}_{i}\}_{i = 1}^{n}$ are
independent random variables that have the same distribution as $Y^{1}_{1}$.
\end{lem}
\begin{proof}
This lemma follows from Section II of
\cite{lindvall2002lectures}. The finiteness of $\mathbb{E}[T^{\beta}]$
follows from Theorem 4.2 of \cite{lindvall2002lectures}. Tracking the
proof, we can see that $\mathbb{E}[T^{\beta}]$ is actually bounded by a
constant times $\mathbb{E}[(Y_{0}^{1})^{\beta}] +
\mathbb{E}[(Y_{0}^{2})^{\beta}]$. 

Let $p_{k} = \mathbb{P}[ Y^{1}_{1} = k]$. Choose a suitable normalizer
$\lambda$ such that 
$$
  c_{k} = \lambda \sum_{ i = k+1}^{\infty} p_{i}, \quad k \geq 0
$$
is a probability distribution. Then the delay distribution $Y^{c}_{0}$
with $\mathbb{P}[ Y^{c}_{0} = k] = c_{k}$ gives a stationary renewal
process (Section II.2 of \cite{lindvall2002lectures}). Further, it is
easy to see that $\mathbb{E}[(Y^{c}_{0})^{\beta - 1}] < \infty $
(Section II.5 of \cite{lindvall2002lectures}). 
\end{proof}

\medskip

Bounds of $\mathbb{E}[(Y^{1}_{0})^{\beta}],
\mathbb{E}[(Y^{2}_{0})^{\beta}]$, and
$\mathbb{E}[(Y^{1}_{1})^{\beta}]$, i.e., $\mathbb{E}_{\mu^{*}}[
\tau_{\alpha}^{\beta}]$, $\mathbb{E}_{\nu^{*}}[
\tau_{\alpha}^{\beta}]$, and $\mathbb{E}_{\alpha}[
\tau_{\alpha}^{\beta}]$, are given in Lemma
\ref{nummelin}. Bounds of $\mathbb{E}[T^{\beta}]$ when starting from
$\mu$ and $\nu$ follow from Lemma
\ref{lindvall}. When starting from $\pi$, passage times to $\alpha$
form a (delayed) stationary renewal process. The delay distribution of
the corresponding renewal process must be $Y^{c}_{0}$, as the delay distribution
that lead to a stationary renewal process is unique
(\cite{barbu2009semi}, Chapter 2).  Therefore, bounds of $\mathbb{E}[T^{\beta}]$ when starting from
$\mu$ and $\pi$ also follow from Lemma \ref{lindvall}. 

\medskip

Theorem \ref{thm31} is then implied by the following simple
probability fact. (See Lemma \ref{m2t}). Let $Z$ be any nonnegative integer-valued
random variable. For any $\beta > 0$,
$$ 
  \mathbb{E}[Z^{\beta}] < \infty \Rightarrow \lim_{n \rightarrow
    \infty} n^{\beta} \mathbb{P}[Z > n] = 0 \,.
$$
\end{proof}

We remark that the bounds in Lemma \ref{nummelin} and Lemma
\ref{lindvall} now depend on the corresponding initial
conditions. One needs to track the proof of theorems in
\cite{lindvall2002lectures} and \cite{nummelin1983rate} to see such
dependence. The dependence on initial conditions implies the following
corollary, which is used in the proof of {\it Theorem 3}.

\begin{cor}
\label{cor35}
Let $\Psi_{n}$ and $\mathfrak{C}$ be as in Theorem \ref{thm31}, then
for any probability measures $\mu, \nu$ on $X$ that satisfy 
$$
  \mathbb{E}_{\mu}[ \tau_{\mathfrak{C}}^{\beta}] < \infty \quad \mathrm{and}
  \quad  \mathbb{E}_{\nu}[ \tau_{\mathfrak{C}}^{\beta}] < \infty \,,
$$ 
there exists a constant $C$ depending on $\mu, \nu,$ and
$\mathfrak{C}$ such that
$$
  \sup_{n} n^{\beta}\| \mu \mathcal{P}^{n} - \nu
  \mathcal{P}^{n} \|_{TV} \leq C ( \mathbb{E}_{\mu}[
  \tau_{\mathfrak{C}}^{\beta}] +\mathbb{E}_{\nu}[ \tau_{\mathfrak{C}}^{\beta}]) \,.
$$
\end{cor}
\begin{proof}
We have
$$
  \sup_{n} n^{\beta}\| \mu \mathcal{P}^{n} - \nu
  \mathcal{P}^{n} \|_{TV} \leq 2 \sup_{n} n^{\beta} \mathbb{P}[ T > n]
  \leq 2 \mathbb{E}_{\mu^{*}, \nu^{*}}[ T^{\beta}]. 
$$
The corollary then follows immediately from Lemma \ref{nummelin} and
Lemma \ref{lindvall}. 
\end{proof}

\subsection{Induced chain method}

As discussed above, a crucial step towards polynomial ergodicity is to
estimate the moments of $\tau_{\mathfrak{C}}$, i.e., the first passage time to a certain uniform
reference set. In some simple cases such as the model in \cite{li2016polynomial},
this can be done by constructing a Lyapunov function. However, in our
model and many other problems, it is extremely difficult to find a single Lyapunov function to
complete this task. Here, we introduce a method, called the induced chain
method, that can be used to estimate the moments of $\tau_{\mathfrak{C}}$ under
more general settings.

Let $X = G \cup B$ be a partition of the state space of $\Psi_{n}$,
where $G$ is the ``good set'' on which $\Psi_{n}$ is sufficiently ``active'', while
$B$ is the ``bad set'' on which $\Psi_{n}$ may hover for a long
time. Define $0 = T_{0} < T_{1} < \cdots < T_{n} < \cdots$ to be return
times to $G$ such that 
$$
  T_{n+1} = \inf\{ k > T_{n} \,|\, \Psi_{k} \in G \} 
$$
and let $\hat{\Psi}_{n} = \Psi_{T_{n}}$. Then it is easy to check
that $\hat{\Psi}_{n}$ is a Markov chain induced by $G$.

Assume $\mathfrak{C} \subset G$. The tail of $\tau_{\mathfrak{C}}$ can
be estimated by the
following two assumptions.
\begin{itemize}
  \item[(i)] $T_{n+1} - T_{n}$ has a polynomial tail for each
    $T_{n}$. There exists a constant $\alpha > 0$ such that
$$
  \mathbb{P}[ T_{n+1} - T_{n} > k \,|\, \Psi_{T_{n}} ] \leq \xi(\Psi_{T_{n}}) k^{-\alpha} \,,
$$
where $1 \leq \xi(\Psi_{T_{n}}) < \infty$ is a constant depending on
$\Psi_{T_{n}}$. Furthermore, $\xi(\Psi_{T_{n}})$ is uniformly bounded by
$\xi_{1} < \infty$ for $\Psi_{T_{n}} \in G$.
\item[(ii)] $\hat{\tau}_{\mathfrak{C}} = \inf \{ n > 0 \,|\, \hat{\Psi}_{n} \in \mathfrak{C} \}$ has
  an exponential tail. There exist constants $\omega$, $\eta  = \eta( \Psi_{0}) > 0$ such that
$$
  \mathbb{P}_{\Psi_{0}}[ \hat{\tau}_{\mathfrak{C}} > k] \leq
  \eta(\Psi_{0}) e^{-\omega k} \,,
$$
where $\eta(\Psi_{0})$ depends on $\Psi_{0}$. 
\end{itemize}

\begin{thm}
\label{inducedchain}
Assuming (i) and (ii) above, for any $\epsilon > 0$, there exists a
constant $c$ such that 
$$
  \mathbb{P}_{\Psi_{0}}[ \tau_{\mathfrak{C}} > n] \leq c( \eta(\Psi_{0}) +
  \xi(\Psi_{0})) n^{-(\alpha - \epsilon)} 
$$
for any $\Psi_{0} \in X$.
\end{thm}
\begin{proof}
For any small $\delta > 0$, we have
$$
  \{ \tau_{\mathfrak{C}} > k^{1 + \delta} \} \subset \{ \hat{\tau}_{\mathfrak{C}} >
  k^{\delta} \} \bigcup_{n = 0}^{\left \lfloor k^{\delta} \right
    \rfloor} \{ T_{n+1} - T_{n} > k, \hat{\tau}_{\mathfrak{C}} > n \} \,.
$$
This implies
\begin{eqnarray*}
\mathbb{P}_{\Psi_{0}}[\tau_{\mathfrak{C}} > k^{1+\delta}] &\leq  &
\mathbb{P}_{\Psi_{0}}[\hat{\tau}_{\mathfrak{C}} > k^{\delta}] \\
&&+\sum_{n = 0}^{\left \lfloor k^{\delta} \right \rfloor}
\mathbb{P}_{\Psi_{0}}[ T_{n+1} - T_{n} > k\,|\, \hat{\tau}_{\mathfrak{C}} > n] \,.
\end{eqnarray*}

By assumption (i) and the Markov property, we have
$$
  \mathbb{P}_{\Psi_{0}}[T_{1} - T_{0} > k \,|\, \hat{\tau}_{\mathfrak{C}} > 0]
  \leq \xi(\Psi_{0}) k^{-\alpha}
$$
for $n = 0$ and
\begin{eqnarray*}
 & & \mathbb{P}_{\Psi_{0}}[ T_{n+1} - T_{n} > k\,|\, \hat{\tau}_{\mathfrak{C}} >
  n]  \\
&=& \int_{G} \mathbb{P}[T_{n+1} - T_{n} > k \,|\,
\Psi_{T_{n} } = x , \hat{\tau}_{\mathfrak{C}} > n]
\mathbb{P}_{\Psi_{0}}[ \Psi_{T_{n}} = \mathrm{d}x \, | \,
\hat{\tau}_{\mathfrak{C}} > n ]\\
&\leq& \xi_{1} k^{-\alpha} 
\end{eqnarray*}
for $n \geq 1$.

By assumption (ii), we have
$$
  \mathbb{P}_{\Psi_{0}}[ \hat{\tau}_{\mathfrak{C}} > k^{\delta} ] \leq
  \eta(\Psi_{0}) e^{-\omega k^{\delta} } \,.
$$ 

Therefore,
\begin{eqnarray*}
&&  \mathbb{P}_{\Psi_{0}}[ \tau_{\mathfrak{C}} > k^{1+\delta}] \\
&\leq& \eta(\Psi_{0}) e^{-\omega k^{\delta}} + \max \{ \xi(\Psi_{0}), \xi_{1} \} k^{\delta} k^{-\alpha}  \\
&\leq& c(\delta) ( \eta(\Psi_{0}) + \xi(\Psi_{0}) )
k^{-(\alpha -\delta)} 
\end{eqnarray*}
for some $c(\delta) > 0$ that depends on $\delta$, as
$e^{-\omega k^{\delta}}$ decays faster than $k^{-\alpha}$ when $k \rightarrow
\infty$.

Let $k = n^{\frac{1}{1 + \delta}}$. The theorem then follows by making $\delta > 0$ sufficiently small and
letting $c = c(\delta)$ for the $\delta$ we choose.
\end{proof}

\medskip

The finiteness of moments of a random variable is closely related to
its polynomial tails. We finish this subsection by proving two simple
probabilistic facts that will be used frequently in this paper. 

Let $Z$ be a random variable that takes nonnegative integer values, and let
$\beta > 1$. 

\begin{lem} 
\label{m2t}
For any $\beta > 0$, $\mathbb{E}[Z^\beta]<\infty \  \implies \ 
\lim_{n\to \infty} n^\beta \mathbb{P}[Z>n] = 0$.
\end{lem}
\begin{proof} 

First,
 \begin{eqnarray*}
  \sum_{n = 0}^{\infty} n^{\beta - 1} \mathbb{P}[ Z > n] & = & \sum_{n
    = 0}^{\infty} n^{\beta - 1}\sum_{m = n+1}^{\infty} \mathbb{P}[ Z = m]\\
&  = & \sum_{m = 1}^{\infty} \left( \sum_{n = 0}^{m-1} n^{\beta - 1} \right
  ) \mathbb{P}[ Z = m] \leq \mbox{constant} \cdot \mathbb{E}[ Z^{\beta}] \ .
\end{eqnarray*}
Then, letting $a_n = \mathbb{P}[ Z > n]$ so that $a_0 \ge a_1 \ge \dots$, we claim that
$$
\mbox{ if } \quad \sum_{n = 0}^{\infty} a_n n^{\beta - 1}  < \infty  \quad \mbox{ then } \quad
a_n n^\beta \to 0 \ \mbox{ as } \  n \to \infty .
$$ 
To see that, write
$$
\sum_{n = 0}^{\infty} (n+1)^{\beta} (a_{n} - a_{n+1}) = 
\sum_{n = 0}^{\infty} \left [ (n+1)^{\beta} - n^{\beta}
\right ] a_{n} \ .
$$
Since $a_{n}$ is monotone, our hypothesis implies the sum on the left side converges.
Since 
$$
\sum_{n \ge N}  (n+1)^{\beta} (a_{n} - a_{n+1}) = a_N (N + 1)^\beta + 
\sum_{n \ge N+1} a_n ( (n + 1)^\beta - n^\beta)\ ,
$$
it follows that both terms on the right tend to $0$ as $N \to \infty$.
\end{proof}

\begin{lem}
\label{t2m}
If $\mathbb{P}[Z > n] \le C (n + 1)^{-\beta}$ for $\beta > 1$, then
for any $\beta - 1 > \varepsilon >
0$, there exists a constant $K$ that depends on $\beta$ and $\varepsilon$, such
that $\mathbb{E}[Z^{\beta - \varepsilon}] < KC$.
\end{lem}

\begin{proof} 
 There exists a constant $K_{0}$ that only depends on $\epsilon$ and
 $\beta$, such that
\begin{eqnarray*}
\mathbb{E}[Z^{\beta - \varepsilon}] &
=  &\sum_{n = 0}^{\infty} n^{\beta - \varepsilon}\mathbb{P}[Z = n] \\
&\le& K_{0} \sum_{n = 0}^{\infty} \sum_{m = 0}^{n} m^{\beta - 1 -\varepsilon} 
\mathbb{P}[Z = n]\\
& = &  K_{0} \sum_{n = 0}^{\infty} (n + 1)^{\beta - 1 -\varepsilon} \mathbb{P}[Z > n] \\
& \le & K_{0} \sum_{n = 0}^{\infty}
(n + 1)^{\beta - 1 -\varepsilon}  C (n + 1)^{-\beta}.
\end{eqnarray*}

Hence there exists a constant $K$ that depends on $\beta$ and
$\epsilon$, such that the last quantity is less than $KC$.
\end{proof}

\medskip

We remark that the induced
chain method can be extended to a hierarchical setting. Let $X = B
\cup G$ be the same partition as above. If we have $\mathfrak{C} := G_{m}
\subset G_{m - 1} \subset \cdots \subset G_{0} := G$ for $m \geq 1$,
and the first return time to $G_{i+1}$ for the $G_{i}$-induced chain
has exponential tail for each $i = 0 \sim m - 1$, then the similar argument in
Theorem \ref{inducedchain} still follows.

\subsection{Lyapunov function and moments of the return time}

In the induced chain argument above, it remains to find sufficient
conditions to estimate tails of $\hat{\tau}_{\mathfrak{C}}$ and $T_{n+1} -
T_{n}$. This can be done by constructing Lyapunov-type functions. We
introduce the following two theorems that will be used later.

\begin{thm}
\label{expfirstpassage}
(Theorem 15.2.5 of \cite{meyn2009markov})

Let $\Psi_{n}$ be a Markov chain on
$(X, \mathcal{B})$ with transition kernel $\mathcal{P}$. We assume
that there exist a function $W: X \rightarrow [1, \infty]$, a set $A
\in \mathcal{B}$, constants $b > 0$ and $0 \leq \beta < 1$ such
that 
$$
  \mathcal{P}W - W \leq -\beta W + b \mathbf{1}_{A} \,.
$$
Then for any $r \in (1, (1 - \beta)^{-1} )$, there exists $\epsilon >
0$ such that 
$$
  \mathbb{E}_{x}\left [ \sum_{k = 0}^{\tau_{A} - 1} W( \Psi_{k})
    r^{k}\right] \leq \epsilon^{-1}r^{-1} W(x) + \epsilon^{-1} b
  \mathbf{1}_{A} \,.
$$
\end{thm}

\begin{thm}
\label{polyfirstpassage}
(Modified from Theorem 3.6 of \cite{jarner2002polynomial})Let $\Psi_{n}$ be a Markov chain on
$(X, \mathcal{B})$ with transition kernel $\mathcal{P}$. We assume
that there exist a function $W: X \rightarrow [1, \infty]$, a set $A
\in \mathcal{B}$, constants $b, c > 0$ and $0 \leq \beta < 1$ such
that 
$$
  \mathcal{P}W - W \leq -c W^{\beta} + b \mathbf{1}_{A} \,.
$$
Then there exists a constant $\hat{c}$ such that 
$$
  \mathbb{E}_{x}\left [ \sum_{k = 0}^{\tau_{A} - 1} (k+1)^{\hat{\beta}
    - 1}\right ] \leq \hat{c} W(x), \quad \hat{\beta} = (1 -
\beta)^{-1} 
$$
for any $x \in X$
\end{thm}
\begin{rem}
This theorem follows from equation (37) in the
proof of Theorem 3.6 of \cite{jarner2002polynomial}. It is actually a
special case of equation (37) when $B = C$. In this special case, the quantity 
$$
  \sum_{k = 0}^{\tau_{B - 1}}(k+1)^{i} \mathbf{1}_{C}(\Psi_{k})
$$
in the proof of Theorem 3.2 of \cite{jarner2002polynomial} is $\mathbf{1}_{C}(\Psi_{0})$. Therefore, we only need
to estimate the first passage time to $A$ by using Proposition 11.3.3
of \cite{meyn2009markov}. (See the proof of Theorem 3.2 of
\cite{jarner2002polynomial} for details.) The original theorem in
\cite{jarner2002polynomial} estimates
$$
  \mathbb{E}_{x}\left [ \sum_{k = 0}^{\tau_{B} - 1} (k+1)^{\hat{\beta}
    - 1}\right ]
$$
for {\it any} set $B$, hence more assumptions are needed than in our
case. 
\end{rem}

\section{Excursion time on low energy set}
\label{low}

It is obvious that for $\mathbf{E}_{t}$, the ``bad set'' $B$ (see
Section 3.2) should consist of energy configurations at which at least
one of $E_{i}$ is sufficiently small. It remains to estimate the
excursion time on this ``bad set''. To do so, we define the following
sequence and function. Let $0 < \eta \ll 1$ be a parameter that will be
determined later. Let $a_{1}, \cdots, a_{N}$ be the sequence as in
Section 2.3:
$$
  a_{i} = 1 - \frac{2^{i-1} - 1}{2^{N} - 1} \,.
$$
Then it is easy to see that $1 = a_{1} > a_{2} >\cdots > a_{N} > 0$
and $2 a_{i-1} - a_{i} > a_{1}$ for each $2 \leq i \leq N$.
 
Define functions
$$
  V_{n, k} = ( \sum_{j = 0}^{n-1}
    E_{k + j}  )^{a_{n}\eta - 1}
$$
$$
  V_{1}(\mathbf{E}) = \sum_{i = 1}^{N} V_{1,i} = \sum_{i = 1}^{N}
  E_{i}^{a_{1}\eta - 1}
$$
$$
  V_{2}(\mathbf{E}) =\sum_{i = 1}^{N - 1}  V_{2, i} =  \sum_{ i =
    1}^{N-1} (E_{i} + E_{i+1})^{a_{2}\eta - 1}
$$
$$
  \vdots
$$
$$
  V_{n}( \mathbf{E}) = \sum_{k = 1}^{N - n + 1} V_{n, k} = \sum_{k =
    1}^{N - n + 1}\left ( \sum_{j = 0}^{n-1}
    E_{k + j} \right )^{a_{n}\eta - 1}
$$
$$
  \vdots
$$
$$
  V_{N}( \mathbf{E}) = V_{N, 1} =  (\sum_{i = 1}^{N}
  E_{i})^{a_{N} \eta - 1} \,,
$$
and
$$
  V( \mathbf{E}) = \sum_{i = 1}^{N} V_{i}( \mathbf{E}) \,.
$$
The motivation of constructing $V$ is to control the entire chain
through nearest neighbor interactions. $V_{1, i}(E)$ is the ``first level
Lyapunov function'' with respect to $E_{i}$, whose value decreases when $E_{i}$
increases. $V_{2, i}$ is the ``second level'', which is dominantly
larger than the ``first level'' functions $V_{1, i}$ and $V_{1, i+1}$,
such that its decrease can compensate the possible increase of $V_{1, i}$ and
$V_{1, i+1}$. Higher levels can be constructed analogously. This tower construction of Lyapunov functions stops at
the $N$-th level that covers the entire chain.

Our aim is to show that $V( \mathbf{E})$ is a Lyapunov function when
the value of $V( \mathbf{E})$ is sufficiently large. The excursion
time on low energy set then follows from Theorem
\ref{polyfirstpassage}. 

The main theorem in this section is as follows.

\begin{thm}
\label{timestep}
For any $\eta > 0$ and $h > 0$ small enough, there exist $c_{0} > 0$,
$M_{0} > 1$ depending on $\eta$, $N$, and $h$, such that
$$
  (P^{h})V( \mathbf{E}) - V(\mathbf{E}) \leq -c_{0} V^{\alpha}(\mathbf{E})
$$
for every $\mathbf{E} \in \{ V > M_{0} \}$, where $\alpha = 1 -
\frac{1}{2(1 - \eta)}$.  
\end{thm}

Let $\mathbf{E} = (E_{1}, \cdots, E_{N})$ be the initial condition. For the sake of simplicity, we let $R_{i} = R(E_{i - 1},
E_{i})$ for all $i = 1, \cdots, N+1$. $\mathbf{E}_{t}$ is given by the
following equivalent description: Starting from $t = 0$, a clock rings
at an exponentially distributed random time $\tau_{1}$ with rate
$\sum_{i = 0}^{N} R_{i}$. When the clock rings, one and only one
energy exchange occurs between site $i$ and $i+1$ with probability 
$$
  \frac{R_{i}}{\sum_{i = 0}^{N} R_{i}} := R_{i}/ \mathcal{R} \,.
$$
In other words, $\mathbb{P}_{\mathbf{E}}[ \mathcal{C}_{i}(\tau_{1})] = R_{i}/ \mathcal{R}$.

Let $\mathbf{E}_{\tau_{1}^{+}}$ be the energy
configuration immediately after the first energy exchange occurs at
$\tau_{1}$. We use $V(\mathbf{E}_{\tau_{1}^{+}})$ to estimate
$P^{h}V(\mathbf{E})$. The strategy of proving Theorem \ref{timestep} is as follows. In Lemma
4.2 and Corollary 4.3, we show that an energy exchange event can
increase a $V_{n, k}$ by at most multiplying with a constant. However, if
$V_{n, k}$ is sufficiently large and the energy stored at its
``boundary site'', i.e., $E_{k-1}$ or $E_{k+n}$, is sufficiently
large, then in Lemma 4.4 we prove that the corresponding energy
exchange event will reduce the value of such a $V_{n, k}$ by at least
one half. Then in Lemma 4.5 we manage to prove that if
$V_{n,k}$ is sufficiently large, the ``expected jump'' of $V_{n,k}$ at
$\tau_{1}$ can be compensated by the ``expected drop'' of $V_{n', k'}/4N^{2}$ for
{\it some} $V_{n', k'}$. This can always be achieved because $V_{n', k'}$
is significantly greater than $V_{n, k}$ when $V_{n', k'}$ ``covers''
$V_{n, k}$. Finally, in Lemma 4.6 we prove that the ``expected drop''
of $V$ at $\tau_{1}$ is proportional to $V^{\alpha}$ for $\alpha = 1 -
\frac{1}{2(1 - \eta)}$. The heuristics of Lemma 4.6 is that for every
sufficiently large $V_{n, k}$, the ``expected jump'' is compensated by
$1/4N^{2}$ of the ``expect drop'' of {\it some} $V_{n', k'}$ (Lemma
4.5). Therefore, at $\tau_{1}$ the total ``expected jump'' of those large $V_{n, k}$s is compensated by
$1/4$ of the largest ``expected drops'' on the left and right side, denoted by $V_{n_{R}, k_{R}}$ and $V_{n_{L},
  k_{L}}$, respectively. Because the ``expected drop'' of at least one $V_{n,
  k}$ is proportional to $V^{\alpha}$, we know that the ``expected drop''
contributed by $V_{n_{R}, k_{R}}$ and $V_{n_{L}, k_{L}}$ minus the
total ``expected jump'' of all large $V_{n,k}$s, is also proportional to
$V^{\alpha}$. On the other hand, Lemma 4.2 implies the ``expected
jump'' of small $V_{n,k}$s at $\tau_{1}$ can be controlled by a
constant. 

\medskip

\begin{lem}
\label{growthlemma}
For any $\mathbf{E} \in \mathbb{R}^{N}_{+}$, $1\leq n \leq N$, $1 \leq k
\leq N - n + 1$, and  $0 < \eta < \frac{1}{2}$, we have
$$
  \mathbb{E}_{\mathbf{E}}[ V_{n,k}( \mathbf{E}_{\tau_{1}^{+}}) \mathbf{1}_{\mathcal{C}_{k}(\tau_{1}) }]
  \leq \frac{C}{\mathcal{R}} \left ( \sum_{ i = 0}^{n-1}E_{k+i}
  \right )^{a_{n}\eta - \frac{1}{2}}
$$
and 
$$
  \mathbb{E}_{\mathbf{E}}[ V_{n,k}( \mathbf{E}_{\tau_{1}^{+}}) \mathbf{1}_{\mathcal{C}_{k + n}(\tau_{1})} ]
  \leq \frac{C}{\mathcal{R}} \left ( \sum_{ i = 0}^{n-1}E_{k+i}
  \right )^{a_{n}\eta - \frac{1}{2}}
$$
for some $C > 0$ depending on $N$ and $\eta$. 
\end{lem}
\begin{proof}
At the first energy exchange, if $k \neq 1$ and $k+n \neq N + 1$, we have
\begin{eqnarray*}
  &&\mathbb{E}_{\mathbf{E}}[V_{n, k}( \mathbf{E}_{\tau_{1}^{+}})
  \mathbf{1}_{\mathcal{C}_{k}(\tau_{1})} ]  =
  \mathbb{E}_{\mathbf{E}}[V_{n, k}( \mathbf{E}_{\tau_{1}^{+}}) \,|\,
  \mathcal{C}_{k}(\tau_{1})]
  \mathbb{P}_{\mathbf{E}}[ \mathcal{C}_{k}(\tau_{1})]\\
&=&
  \frac{\min\{ K, \sqrt{ \min \{E_{k-1}, E_{k}\}} \}}{\mathcal{R}}\int_{0}^{1}  \left [ p
    (E_{k- 1} + E_{k}) + E_{k+1 }
  \cdots + E_{k + n - 1} \right ]^{ a_{n} \eta - 1} \mathrm{d} p := I_{1}
\end{eqnarray*}
Notice that $a_{n}\eta$ is small so $a_{n}\eta - \frac{1}{2} < 0$. If $E_{k} \geq \frac{1}{2} \sum_{i = 1}^{n-1} E_{k+i}$, then
$$
  I_{1} \leq
  \frac{E_{k}^{1/2}}{\mathcal{R}} \cdot E_{k}^{a_{n}\eta - 1}\int_{0}^{1}
  p^{a_{n}\eta - 1} \mathrm{d}p \leq \frac{2}{a_{n}\eta \mathcal{R}} \left ( \sum_{ i = 0}^{n-1}E_{k+i}
  \right )^{a_{n}\eta - \frac{1}{2}} \,.
$$
Otherwise,
\begin{eqnarray*}
    &&I_{1}   \leq \frac{E_{k}^{1/2}}{\mathcal{R}} \cdot (E_{k+1} + \cdots +
    E_{k+n - 1})^{a_{n}\eta - 1}\\
  &\leq& 2 \frac{(E_{k} + \cdots + E_{k+n - 1})^{1/2}}{\mathcal{R}} \cdot (E_{k}  + \cdots +
    E_{k+n - 1})^{a_{n}\eta - 1}  \leq \frac{2}{ \mathcal{R}} \left ( \sum_{ i = 0}^{n-1}E_{k+i}
  \right )^{a_{n}\eta - \frac{1}{2}} \,.
\end{eqnarray*}
The same argument holds for 
\begin{eqnarray*}
&&  \mathbb{E}_{\mathbf{E}}[V_{n, k}( \mathbf{E}_{\tau_{1}^{+}}) \mathbf{1}_{\mathcal{C}_{k+n}(\tau_{1})} ] \\
&=&
  \frac{\min \{ K, \sqrt{ \min \{E_{k+n-1}, E_{k+n}\}} \}}{\mathcal{R}}\int_{0}^{1}  \left [ E_{k} 
  \cdots + p( E_{k + n - 1} + E_{k+n}) \right ]^{ a_{n} \eta - 1} \mathrm{d} p 
\end{eqnarray*}
by discussing cases $E_{k+n - 1} \geq
\frac{1}{2} \sum_{i = 1}^{n-2} E_{k+i}$ and $E_{k+n - 1} < \frac{1}{2}
\sum_{i = 1}^{n-2} E_{k+i}$. 

If $V_{k,n}$ involves the boundary, say $k
= 1$, then 
\begin{eqnarray*}
&&\mathbb{E}_{\mathbf{E}}[V_{n, k}( \mathbf{E}_{\tau_{1}^{+}}) \mathbf{1}_{\mathcal{C}_{k}(\tau_{1}) }] \\
  &=& \frac{\sqrt{ \min \{T_{L},
      E_{1}\}}}{\mathcal{R}} \int_{0}^{\infty} \int_{0}^{1} \frac{1}{T_{L}} \left [ p
    (x + E_{1}) + E_{2 }
  \cdots + E_{n} \right ]^{ a_{n} \eta - 1} e^{-x/T_{L}}\mathrm{d} p
\mathrm{d}x \,.
\end{eqnarray*}
It is easy to check that the same argument above still holds. The case
of $k+n = N+1$ can be estimated analogously. 

The proof is completed by letting $C = \frac{2}{a_{N}\eta} \geq \frac{2}{a_{n}\eta}$.
\end{proof}

The calculation in the proof of Lemma \ref{growthlemma} gives the
following Lemma.

\begin{lem}
\label{cor1}
For any $\mathbf{E} \in \mathbb{R}^{N}_{+}$, $\eta > 0$, $1\leq n \leq N$, and $1 \leq k
\leq N - n + 1$, 
$$
  \mathbb{E}_{\mathbf{E}}[ V_{n,k}( \mathbf{E}_{\tau_{1}^{+}}) \,|\,
  \mathcal{C}_{k}(\tau_{1})]
  \leq C V_{n,k}( \mathbf{E}) 
$$
and 
$$
  \mathbb{E}_{\mathbf{E}}[ V_{n,k}( \mathbf{E}_{\tau_{1}^{+}}) \,|\, \mathcal{C}_{k
    + n}(\tau_{1})  ]
  \leq C V_{n,k}( \mathbf{E}) \,,
$$
where $C > 0$ is the same as that in Lemma
\ref{growthlemma}. 
\end{lem}

\begin{proof}
We have
$$
  \mathbb{E}_{\mathbf{E}}[ V_{n,k}( \mathbf{E}_{\tau_{1}^{+}}) \,|\,
  \mathcal{C}_{k}(\tau_{1}) ] = \int_{0}^{1}  \left [ p
    (E_{k- 1} + E_{k}) + E_{k+1 }
  \cdots + E_{k + n - 1} \right ]^{ a_{n} \eta - 1} \mathrm{d} p \,.
$$
It follows from the same argument as in the proof of Lemma
\ref{growthlemma} that
\begin{eqnarray*}
 & & \int_{0}^{1}  \left [ p
    (E_{k- 1} + E_{k}) + E_{k+1 }
  \cdots + E_{k + n - 1} \right ]^{ a_{n} \eta - 1} \mathrm{d} p \\
&\leq & \max\{ \frac{2}{a_{n}\eta} , 2 \} (E_{k} + \cdots + E_{k+n -
  1} )^{a_{n} \eta - 1} := C V_{n,k}( \mathbf{E}) \,.
\end{eqnarray*}
The proof of 
$$
   \mathbb{E}_{\mathbf{E}}[ V_{n,k}( \mathbf{E}_{\tau_{1}^{+}}) \,|\, \mathcal{C}_{k
    + n}(\tau_{1})  ]
  \leq C V_{n,k}( \mathbf{E})
$$
is similar.
\end{proof}

\begin{lem}
\label{controllemma}
For any $\mathbf{E} \in \mathbb{R}^{N}_{+}$ and any $0 < \eta \ll 1$,
there exists a constant $C'$ depending on $N$, $T_{L}$, $T_{R}$, and
$\eta$, such that whenever $E_{k+n} > C'(E_{k} +
\cdots + E_{k+n - 1})$ (resp. $E_{k- 1} > C'(E_{k} +
\cdots + E_{k+n - 1})$), we have
$$
  \mathbb{E}_{\mathbf{E}}[V_{n, k}( \mathbf{E}_{\tau_{1}^{+}}) \,|\,  \mathcal{C}_{k + n }(
  \tau_{1})] < \frac{1}{2} V_{n, k}(\mathbf{E}) 
$$
( resp. 
$$
  \mathbb{E}_{\mathbf{E}}[V_{n, k}( \mathbf{E}_{\tau_{1}^{+}}) \,|\, \mathcal{C}_{k  }(\tau_{1})] <
  \frac{1}{2} V_{n, k}(\mathbf{E})  \,.
$$
)

\end{lem}
\begin{proof}
First assume $k + n \neq N+1$, then we have
\begin{eqnarray*}
  &&\mathbb{E}_{\mathbf{E}}[V_{n, k}( \mathbf{E}_{\tau_{1}^{+}}) \,|\, \mathcal{C}_{k + n
  }(\tau_{1})]\\ &=& \int_{0}^{1 } ( E_{k} + \cdots + E_{k+n - 2}
  + p(E_{k+n - 1} + E_{k+n}))^{a_{n}\eta - 1} \mathrm{d}p\\
& <&
  E_{k+n}^{a_{n}\eta - 1}\int_{0}^{1} p^{a_{n}\eta - 1} \mathrm{d}p\,. 
\end{eqnarray*}
If $E_{k+n} > C'(E_{k} + \cdots + E_{k+n - 1})$, we have
$$
  E_{k+n}^{a_{n}\eta - 1}\int_{0}^{1} p^{a_{n}\eta - 1} \mathrm{d}p
  \leq \frac{C'^{a_{n}\eta - 1}}{a_{n}\eta} V_{n, k}(\mathbf{E}) \,.
$$
To make 
$$
  \frac{C'^{a_{n}\eta - 1}}{a_{n}\eta} \leq \frac{1}{2} \,,
$$
one needs
$$
  C' \geq (\frac{1}{2} a_{n} \eta)^{\frac{1}{a_{n} \eta - 1}} \,.
$$

If $k + n = N+1$, then
\begin{eqnarray*}
  &&\mathbb{E}_{\mathbf{E}}[V_{n, k}( \mathbf{E}_{\tau_{1}^{+}}) \,|\, \mathcal{C}_{k + n
  }(\tau_{1})]\\ &=& \int_{0}^{1 } \int_{0}^{\infty} \frac{1}{T_{R}}( E_{k} + \cdots + E_{k+n - 2}
  + p(E_{k+n - 1} + x))^{a_{n}\eta - 1} e^{-x/T_{R}} \mathrm{d}p \mathrm{d}x\\
& <&
  \int_{0}^{\infty} \frac{1}{T_{R}}x^{a_{n}\eta - 1} e^{-x/T_{R}}
  \mathrm{d}x\int_{0}^{1} p^{a_{n}\eta - 1} \mathrm{d}p\\
&=& T_{R}^{a_{n}\eta - 1} \frac{\Gamma( a_{n} \eta)}{a_{n}\eta} \,,
\end{eqnarray*}
where $\Gamma(\cdot)$ is the Gamma function. Therefore, to make
$$
  \mathbb{E}_{\mathbf{E}}[V_{n, k}( \mathbf{E}_{\tau_{1}^{+}}) \,|\, \mathcal{C}_{k + n }(\tau_{1})]
  < \frac{1}{2} V_{n, k}( \mathbf{E}) \,,
$$
we need
$$
  V_{n, k}(\mathbf{E}) = (E_{k} + \cdots + E_{k+n - 1})^{a_{n}\eta
    - 1} > 2 T_{R}^{a_{n}\eta - 1} \frac{\Gamma( a_{n}
    \eta)}{a_{n}\eta} \,.
$$
Since $E_{N+1} = T_{R}$, this is equivalent to 
\begin{eqnarray*}
  E_{N+1}& >& T_{R} \left [ 2 T_{R}^{a_{n}\eta - 1} \cdot
    \frac{\Gamma(a_{n}\eta)}{a_{n}\eta} \right]^{-\frac{1}{a_{n}\eta -
    1}} (E_{k} + \cdots + E_{k+n - 1})  \\
&&= \left (\frac{2\Gamma(a_{n}\eta) }{a_{n}\eta
}\right)^{-\frac{1}{a_{n}\eta -1}}(E_{k} + \cdots + E_{k+n - 1}) \,.
\end{eqnarray*}
The case for $\mathcal{C}_{k}$ is symmetric and can be
calculated in the same way. By combining all cases, it is easy to
check that if
$$
  C' = \max_{1 \leq n \leq N} \{ ( \frac{1}{2} a_{n}\eta
  )^{\frac{1}{a_{n}\eta - 1}}, \left (\frac{2\Gamma(a_{n}\eta) }{a_{n}\eta
}\right)^{-\frac{1}{a_{n}\eta -1}} \} \,,
$$
we have the desired property for all $n$ and $k$.

\end{proof}

\begin{lem}
\label{controllemma2}
For any $\mathbf{E} \in \mathbb{R}^{N}_{+}$ and any $ 0 < \eta \ll 1$,
there exists an $M < \infty$ depending on $\eta, T_{L}, T_{R}$,
and $N$, such that for any $V_{n, k}(\mathbf{E}) > M$,
if 
$$
  \mathbb{E}_{\mathbf{E}}[ V_{n, k}(\mathbf{E}_{\tau_{1}^{+}}) \mathbf{1}_{
  \mathcal{C}_{k+n}(\tau_{1})}] \geq V_{n, k}(\mathbf{E})
  \mathbb{P}_{\mathbf{E}}[ \mathcal{C}_{k+n}(\tau_{1})]
$$
(resp.
$$
  \mathbb{E}_{\mathbf{E}}[ V_{n, k}(\mathbf{E}_{\tau_{1}^{+}})
  \mathbf{1}_{\mathcal{C}_{k}(\tau_{1})}] \geq V_{n, k}(\mathbf{E}) 
  \mathbb{P}_{\mathbf{E}}[\mathcal{C}_{k}(\tau_{1})] \,, 
$$
)
then there exists $k'$ and $n'$, such that 
\begin{eqnarray*}
  &&\mathbb{E}_{\mathbf{E}}[ V_{n, k}(\mathbf{E}_{\tau_{1}^{+}})
  \mathbf{1}_{ \mathcal{C}_{k + n}(\tau_{1})}] - V_{n,
    k}(\mathbf{E}) \mathbb{P }_{\mathbf{E}}[ \mathcal{C}_{k+n}(\tau_{1})]\\
  &\leq& \frac{1}{4 N^{2}} \left \{ V_{n', k'}(\mathbf{E}) \mathbb{P}_{\mathbf{E}}[
    \mathcal{C}_{k' + n'}(\tau_{1})]-  \mathbb{E}_{\mathbf{E}}[
  V_{n', k'}(\mathbf{E}_{\tau_{1}^{+}})  \mathbf{1}_{ \mathcal{C}_{k'+n'}(\tau_{1})}] \right \} \,.
\end{eqnarray*}
(resp.
\begin{eqnarray*}
  &&\mathbb{E}_{\mathbf{E}}[ V_{n, k}(\mathbf{E}_{\tau_{1}^{+}}) \mathbf{1}_{\mathcal{C}_{k}(\tau_{1})}
  ] - V_{n, k}(\mathbf{E}) \mathbb{P}_{\mathbf{E}}[ \mathcal{C}_{k}(\tau_{1})]\\
  &\leq& \frac{1}{4 N^{2}} \left \{ V_{n', k'}(\mathbf{E}) \mathbb{P}_{\mathbf{E}}[
    \mathcal{C}_{k'}(\tau_{1})]-  \mathbb{E}_{\mathbf{E}}[
  V_{n', k'}(\mathbf{E}_{\tau_{1}^{+}})  \mathbf{1}_{\mathcal{C}_{k'}(\tau_{1})}] \right \} \,.
\end{eqnarray*}
)

\end{lem}

\begin{proof}
By symmetry, it is sufficient to consider the case of 
$$
  \mathbb{E}_{\mathbf{E}}[ V_{n, k}(\mathbf{E}_{\tau_{1}^{+}}) \mathbf{1}_{
  \mathcal{C}_{k+n}(\tau_{1})}] \geq V_{n, k}(\mathbf{E})
  \mathbb{P}_{\mathbf{E}}[ \mathcal{C}_{k+n}(\tau_{1})] \,.
$$

Let $k' = k$ and $n'$ be the first $n' > n - 1$ such that 
$$
  E_{k + n'} > C' (E_{k} + \cdots + E_{k + n' - 1} )
$$
where $C'$ is the constant defined in Lemma \ref{controllemma}. When
$M$ is large, the sum $E_{k} + \cdots + E_{k + n' - 1} $ is
small. Note that $E_{N+1} = T_{R}$. Hence when $M$ is sufficiently
large, for any $V_{n, k}(\mathbf{E}_{0}) > M$, one can always find
such an $n'$.

By Lemma \ref{controllemma}, if $n' = n$, the energy exchange event
$\mathcal{C}_{k+n}(\tau_{1})$ will only bring the expected value of
$V_{n, k}$ down. Therefore, it is sufficient to consider the case of
$n' \geq n+1$. The lemma follows if one can prove either
\begin{eqnarray*}
 && \mathbb{E}_{\mathbf{E}}[ V_{n, k}( \mathbf{E}_{\tau_{1}^{+}} ) \mathbf{1}_{
 \mathcal{C}_{k+n}(\tau_{1}) }
  ] - V_{n, k}( \mathbf{E}) \mathbb{P}_{\mathbf{E}}[
  \mathcal{C}_{k+n}(\tau_{1})] \\
& \leq& \frac{1}{4 N^{2}}\left \{ V_{n', k}(\mathbf{E}) \mathbb{P}_{\mathbf{E}}[
    \mathcal{C}_{k+n'}(\tau_{1})]-  \mathbb{E}_{\mathbf{E}}[
  V_{n', k}(\mathbf{E}_{\tau_{1}^{+}})  \mathbf{1}_{ \mathcal{C}_{k+n'}(\tau_{1})}]
\right \} := \frac{1}{4 N^{2}} I_{1}
\end{eqnarray*}
or 
\begin{eqnarray*}
 && \mathbb{E}_{\mathbf{E}}[ V_{n, k}( \mathbf{E}_{\tau_{1}^{+}} ) \mathbf{1}_{\mathcal{C}_{k+n}(\tau_{1})}
  ] - V_{n, k}( \mathbf{E}) \mathbb{P}_{\mathbf{E}}[
  \mathcal{C}_{k+n}(\tau_{1})] \\
& \leq& \frac{1}{4 N^{2}}\left \{ V_{1, k+ n' - 1}(\mathbf{E}) \mathbb{P}_{\mathbf{E}}[
    \mathcal{C}_{k+n'}(\tau_{1})]-  \mathbb{E}_{\mathbf{E}}[
  V_{1, k + n' - 1}(\mathbf{E}_{\tau_{1}^{+}})  \mathbf{1}_{ \mathcal{C}_{k+n'}(\tau_{1})} \right \}  := \frac{1}{4N^{2}} I_{2}
\end{eqnarray*}

Let $x = E_{k} + \cdots + E_{k+n' - 2}$, $y = E_{k+n' - 1}$, and $z =
E_{k+n'}$. Since $a_{n'} < a_{n}$, if $x$ is sufficiently small, by
Lemma \ref{growthlemma}, 
\begin{eqnarray*}
  \mathbb{E}_{\mathbf{E}}[V_{n, k}( \mathbf{E}_{\tau_{1}^{+}}) \mathbf{1}_{\mathcal{C}_{k+n}(\tau_{1})}
  ] &\leq& \frac{C}{\mathcal{R}} (E_{k} +\cdots + E_{k+n-1})^{a_{n}\eta -
    \frac{1}{2}} \\
&\leq& \frac{C}{\mathcal{R}}\left ( \frac{1}{N C'^{N-1}} (E_{k} + \cdots +
E_{k+n' - 2} ) \right )^{a_{n}\eta - \frac{1}{2}}\\
&\leq& \frac{1}{\mathcal{R}} x^{a_{n' - 1}\eta -
    \frac{1}{2}} \cdot \left( C \cdot C'^{(\frac{1}{2} -
      a_{n}\eta)(N-1)} \cdot N^{\frac{1}{2 } - a_{n}\eta} \cdot x^{(a_{n} - a_{n' - 1})\eta}\right ) \\
& \leq& \frac{1}{\mathcal{R}} x^{a_{n' - 1}\eta -
    \frac{1}{2}} \,,
\end{eqnarray*}
where $C$ is the constant in Lemma \ref{growthlemma}. The second inequality above follows from 
$$
  E_{k+m} \leq C'(E_{k} + \cdots + E_{k+m - 1})
$$
for each $n \leq m < n'$. 

The requirement of small $x$ can be satisfied by making $M = M(\eta)$
sufficiently large, because
$$
  x \leq N C'^{N-1} (E_{k} + \cdots + E_{k+n - 1}) \leq N C'^{N-1}
  M^{\frac{1}{a_{n}\eta - 1}} \,.
$$
Similarly, we can make $M = M(\eta)$ sufficiently large such that
$$
  y \leq N C'^{N-1}  M^{\frac{1}{a_{n}\eta - 1}} < K \,.
$$
Notice that $y < z$ because $C'$ in Lemma 4.4 is greater than
$1$ ($C' > (\frac{1}{2} a_{n} \eta)^{ (a_{n} \eta - 1)^{-1}}$, $a_{n}
< 1$, $\eta < 1$). Therefore, we have $R(E_{k + n' - 1}, E_{k + n'}) = \sqrt{y}$. Hence we have 
$$
  I_{1} = \left \{ (x+y)^{a_{n'}\eta - 1} - \int_{0}^{1}[ x +
    p(y+z)]^{a_{n'}\eta - 1} \mathrm{d}p \right \} \frac{\sqrt{y}}{\mathcal{R}}
$$
and
$$
  I_{2} = \left \{ y^{a_{1}\eta - 1} - \int_{0}^{1}[
    p(y+z)]^{a_{1}\eta - 1} \mathrm{d}p \right \}
  \frac{\sqrt{y}}{\mathcal{R}} \,.
$$

Since $z > C' (y + x) > C' y$ and $y < C' x$, by Lemma \ref{controllemma}, we have
$$
  I_{1} \geq \frac{1}{2} (x+y)^{a_{n'}\eta -1}
  \frac{\sqrt{y}}{\mathcal{R}} \geq \frac{1}{2(1 + C')} \frac{x^{a_{n'}\eta -
  1} \sqrt{y}}{\mathcal{R}}
$$
and
$$
  I_{2} \geq \frac{1}{2} y^{a_{1}\eta - 1}
  \frac{\sqrt{y}}{\mathcal{R}} \,.
$$

We claim that 
$$
  \max \left \{ \frac{1}{2} y^{a_{1}\eta - \frac{1}{2}} , \frac{1}{2
      (1 + C')}
  x^{a_{n'}\eta-1} \sqrt{y} \right \}  > 4N^{2} x^{a_{n'-1}\eta
  - \frac{1}{2}} \,.
$$
It is easy to see that the lemma follows from this claim.

{\bf Proof of the claim: } Let $\epsilon = \eta^{2}$. The proof is
splited to two cases. 

{\bf Case 1:} $y < x^{1 + 2 (a_{n' - 1} - a_{n'})\eta -
  \epsilon}$, then
$$
  \frac{1}{2} y^{a_{1}\eta - \frac{1}{2}} > \frac{1}{2} x^{ (a_{1}\eta
  - \frac{1}{2}) [ 1 + 2 (a_{n' - 1} - a_{n'}) \eta - \epsilon]} \,.
$$
Note that $2 a_{n' - 1} - a_{n' } > a_{1}$ and $\eta \ll 1$, we have
\begin{eqnarray*}
 & & (a_{1} \eta - \frac{1}{2}) [ 1 + 2(a_{n' - 1} - a_{n'}) \eta -
 \epsilon] \\
&=& -\frac{1}{2} + a_{n'-1}\eta - \epsilon + ( a_{1} + a_{n'} - 2
a_{n' - 1})\eta + \epsilon + \epsilon(\frac{1}{2} - a_{1}\eta) + 2
a_{1}(a_{n' - 1} - a_{n'})\eta^{2}\\
&<& -\frac{1}{2} + a_{n' - 1} \eta - \epsilon
\end{eqnarray*}
if
$$
  (2 a_{n' - 1} - a_{n'} - a_{1}) \eta > \frac{3}{2}\epsilon +  2a_{1}(a_{n' - 1} - a_{n'})\eta^{2}\,,
$$
which can be achieved by making $\eta$ sufficiently small. (Because
$\epsilon = \eta^{2}$.)

Therefore,
$$
  \frac{1}{2} y^{a_{1}\eta - \frac{1}{2}} > \frac{1}{2}x^{-\epsilon}
  \cdot x^{a_{n'-1}\eta - \frac{1}{2}} > 4N^{2}  x^{a_{n' - 1} \eta -
    \frac{1}{2}}  
$$
when $x$ is sufficiently small, which can be made by letting $M$ large
enough. 

{\bf Case 2: } $y \geq x^{1 + 2 (a_{n' - 1} - a_{n'})\eta -
  \epsilon}$. Recall that we have $\epsilon = \eta^{2}$, we have
\begin{eqnarray*}
\frac{1}{2 C'} x^{a_{n'}\eta - 1} \sqrt{y}& \geq  & \frac{1}{2C'}
x^{a_{n'}\eta - 1} \cdot x^{\frac{1}{2} + a_{n' - 1} \eta - a_{n'}\eta
- \frac{\epsilon}{2}}\\
&=& \frac{1}{2 C' } x^{a_{n' - 1}\eta - \frac{1}{2}} \cdot x^{-\frac{\epsilon}{2}} \\
&>& 4 N^{2}  x^{a_{n' - 1} \eta - \frac{1}{2}} 
\end{eqnarray*}
if $x$ is sufficiently small. Again, this can be achieved by letting
$M$ sufficiently large.
\end{proof}

\begin{lem}
\label{stoppingtime}
For any $\mathbf{E} \in \mathbb{R}^{N}_{+}$ and any $0 < \eta \ll 1$, there exist constants $\alpha = 1 - \frac{1}{2(1 -
\eta)}$, $C_{1}$ and $M'$ depending on
$N$ and $\eta$, such that
$$
  \mathbb{E}_{\mathbf{E}}[ V(\mathbf{E}_{\tau_{1}^{+}})] \leq
  V(\mathbf{E}) - \frac{C_{1}}{\mathcal{R}} V^{\alpha}(
  \mathbf{E}) 
$$
for all $V(\mathbf{E}) > M'$. 
\end{lem}
\begin{proof}
{\bf Step 1. }First we show that there exists $\tilde{n}, \tilde{k}$
such that the ``expected drop'' of $V_{\tilde{n}, \tilde{k}}$ at $\tau_{1}$ is
proportional to $V^{\alpha}(\mathbf{E})$. Let $V_{\tilde{n}, \tilde{k}}( \mathbf{E})$ be the maximum of $\{
V_{i,j}(\mathbf{E})\}$. Therefore $V_{\tilde{n}, \tilde{k}}( \mathbf{E}) \geq
\frac{1}{N(N+1)} V( \mathbf{E})$. Since $V_{\tilde{n}, \tilde{k}}(
\mathbf{E})$ is the maximum, we have
$$
  ( E_{\tilde{k}} + \cdots + E_{\tilde{k} + \tilde{n} - 1}
  )^{a_{\tilde{n}}\eta  - 1} > ( E_{\tilde{k}-1} + \cdots + E_{\tilde{k} + \tilde{n} - 1}
  )^{a_{\tilde{n}+1}\eta  - 1} 
$$
and
$$
  ( E_{\tilde{k}} + \cdots + E_{\tilde{k} + \tilde{n} - 1}
  )^{a_{\tilde{n}}\eta  - 1} > ( E_{\tilde{k}} + \cdots + E_{\tilde{k} + \tilde{n} }
  )^{a_{\tilde{n}+1}\eta  - 1}  \,.
$$
Since $a_{\tilde{n} + 1} < a_{\tilde{n}}$, when $M'$ is
sufficiently large, we have
$$
  E_{\tilde{k} - 1} > C' ( E_{\tilde{k}} + \cdots + E_{\tilde{k} +
    \tilde{n} - 1} ) 
$$
and
$$
  E_{\tilde{k} + \tilde{n}} > C' ( E_{\tilde{k}} + \cdots + E_{\tilde{k} +
    \tilde{n} - 1} )  \,,
$$
where $C'$ is as in Lemma \ref{controllemma}. In addition, we have
$$
  E_{\tilde{k}}^{a_{1}\eta - 1} \leq ( E_{\tilde{k}} + \cdots + E_{\tilde{k} + \tilde{n} - 1}
  )^{a_{\tilde{n}}\eta  - 1}
$$
and
  $$
  E_{\tilde{k} + \tilde{n} - 1}^{a_{1}\eta - 1} \leq ( E_{\tilde{k}} + \cdots + E_{\tilde{k} + \tilde{n} - 1}
  )^{a_{\tilde{n}}\eta  - 1} \,.
$$

Since $M'$ is assumed to be sufficiently large, we have $\min\{K,
\min\{E_{\tilde{k} - 1}, E_{\tilde{k}}\} \} = E_{\tilde{k}}$ and $\min\{K,
\min\{E_{\tilde{k}  + \tilde{n}- 1}, E_{\tilde{k} + \tilde{n}}\} \} =
E_{\tilde{k} + \tilde{n} - 1}$. Therefore,
\begin{eqnarray*}
&&  \mathbb{E}_{\mathbf{E}}[ V_{\tilde{n}, \tilde{k}}(
\mathbf{E}_{\tau_{1}^{+}}) ] - V_{\tilde{n}, \tilde{k}}(
\mathbf{E})\\ 
&=& \frac{\sqrt{E_{\tilde{k}}}}{\mathcal{R}} \left \{ \int_{0}^{1} [
  p(E_{\tilde{k} - 1} + E_{\tilde{k}}) + \cdots + E_{\tilde{n} +
    \tilde{k} - 1} ]^{a_{\tilde{n}}\eta - 1} \mathrm{d} p -
  (E_{\tilde{k}} + \cdots + E_{\tilde{k} + \tilde{n} -
      1})^{a_{\tilde{n}}\eta - 1} \right \}\\
&+& \frac{\sqrt{E_{\tilde{k}+ \tilde{n} - 1} }}{\mathcal{R}} \left \{ \int_{0}^{1} [
  E_{\tilde{k}} + \cdots + p(E_{\tilde{n} +
    \tilde{k} - 1} + E_{\tilde{n} + \tilde{k}}) ]^{a_{\tilde{n}}\eta - 1} \mathrm{d} p -
  (E_{\tilde{k}} + \cdots + E_{\tilde{k} + \tilde{n} -1})^{a_{\tilde{n}}\eta - 1} \right \} \\
&:=& I_{1} + I_{2} \,.
\end{eqnarray*}

It follows from the same calculation as in Lemma \ref{controllemma}
that 
\begin{eqnarray}
\label{controlofii}
  I_{i} &\leq& -\frac{1}{2} ( E_{\tilde{k}} + \cdots + E_{\tilde{n} +
    \tilde{k} - 1} )^{a_{\tilde{n}} \eta - 1} \cdot ( E_{\tilde{k}} + \cdots + E_{\tilde{n} +
    \tilde{k} - 1} )^{\frac{ a_{\tilde{n}}\eta - 1}{2 (a_{1}\eta -
      1)}} \cdot \frac{1}{\mathcal{R}}\\\nonumber
&=& - \frac{1}{2 \mathcal{R}} V_{\tilde{n}, \tilde{k}}^{\alpha}( \mathbf{E})
\end{eqnarray}
for $i = 1, 2$, where
$$
  \alpha = 1 - \frac{1}{2(1 - a_{1}\eta)} \,.
$$
\medskip

{\bf Step 2.} Let $M$ be as in Lemma \ref{controllemma2}. Let $\{
(n_{i}, k_{i})\}_{i = 1}^{m}$ be indices for which $V_{n_{i}, k_{i}}(
\mathbf{E}) \geq M$. The aim of this step is to prove that the total
``expected jump'' of these $V_{n_{i}, k_{i}}$ at $\tau_{1}$ is
dominated by the ``expected drop'' of {\it some} $V_{n_{R}, k_{R}}$
and $V_{n_{L}, k_{L}}$ from left and right side, respectively.

By Lemma \ref{controllemma2}, we can construct sequences $\{(\hat{n}'_{i}, \hat{k}'_{i})\}_{ i = 1}^{m}$ and
$\{(\hat{n}''_{i}, \hat{k}''_{i})\}_{ i = 1}^{m}$ such that if
$$
  \mathbb{E}_{\mathbf{E}}[ V_{n_{i}, k_{i}}( \mathbf{E}_{\tau_{1}^{+}})
\mathbf{1}_{\mathcal{C}_{k_{i} +n_{i}}(\tau_{1})}] \geq V_{n_{i}, k_{i}}(
\mathbf{E}) \mathbb{P}_{\mathbf{E}}[ \mathcal{C}_{k_{i} + n_{i}}( \tau_{1})] \,,
$$
then
\begin{eqnarray*}
&&\mathbb{E}_{\mathbf{E}}[ V_{n_{i}, k_{i}}( \mathbf{E}_{\tau_{1}^{+}})\mathbf{1}_{\mathcal{C}_{k_{i}+n_{i}}(\tau_{1})}] - V_{n_{i}, k_{i}}(
\mathbf{E}) \mathbb{P}_{\mathbf{E}}[ \mathcal{C}_{k_{i} + n_{i}}( \tau_{1})]
\\
&\leq& \frac{1}{4N^{2}} \{ V_{\hat{n}'_{i}, \hat{k}'_{i}}(
\mathbf{E}) \mathbb{P}_{\mathbf{E}}[ \mathcal{C}_{\hat{k}'_{i} + \hat{n}'_{i}}(
\tau_{1})] - \mathbb{E}_{\mathbf{E}}[ V_{\hat{n}'_{i}, \hat{k}'_{i}}( \mathbf{E}_{\tau_{1}^{+}})
\mathbf{1}_{\mathcal{C}_{\hat{k}'_{i}+\hat{n}'_{i}}(\tau_{1})}] \}
\end{eqnarray*}
and if
$$
  \mathbb{E}_{\mathbf{E}}[ V_{n_{i}, k_{i}}( \mathbf{E}_{\tau_{1}^{+}})
  \mathbf{1}_{\mathcal{C}_{k_{i}}(\tau_{1})} ] \geq
    V_{n_{i}, k_{i}}( 
\mathbf{E}) \mathbb{P}_{\mathbf{E}}[ \mathcal{C}_{k_{i}}( \tau_{1})] \,,
$$
then
\begin{eqnarray*}
&&\mathbb{E}_{\mathbf{E}}[ V_{n_{i}, k_{i}}( \mathbf{E}_{\tau_{1}^{+}})
\mathbf{1}_{\mathcal{C}_{k_{i}}(\tau_{1})}] - V_{n_{i}, k_{i}}(
\mathbf{E}) \mathbb{P}_{\mathbf{E}}[ \mathcal{C}_{k_{i}}( \tau_{1})]
\\
&\leq& \frac{1}{4N^{2}} \{ V_{\hat{n}''_{i}, \hat{k}''_{i}}(
\mathbf{E}) \mathbb{P}_{\mathbf{E}}[ \mathcal{C}_{\hat{k}''_{i}}(
\tau_{1})] - \mathbb{E}_{\mathbf{E}}[ V_{\hat{n}''_{i},
  \hat{k}''_{i}}( \mathbf{E}_{\tau_{1}^{+}}) \mathbf{1}_{\mathcal{C}_{\hat{k}''_{i}}(\tau_{1})}] \} \,.
\end{eqnarray*}
When condition
$$
  \mathbb{E}_{\mathbf{E}}[ V_{n_{i}, k_{i}}( \mathbf{E}_{\tau_{1}^{+}})
\mathbf{1}_{\mathcal{C}_{k_{i}+n_{i}}(\tau_{1})}] \geq V_{n_{i}, k_{i}}(
\mathbf{E}) \mathbb{P}_{\mathbf{E}}[ \mathcal{C}_{k_{i} + n_{i}}( \tau_{1})] 
$$
(resp.
$$
  \mathbb{E}_{\mathbf{E}}[ V_{n_{i}, k_{i}}( \mathbf{E}_{\tau_{1}^{+}})
  \mathbf{1}_{\mathcal{C}_{k_{i}}(\tau_{1})} ] \geq
    V_{n_{i}, k_{i}}( 
\mathbf{E}) \mathbb{P}_{\mathbf{E}}[ \mathcal{C}_{k_{i}}( \tau_{1})] 
$$
)
is not satisfied, we simply let $\hat{n}'_{i} = n_{i}, \hat{k}'_{i} =
k_{i}$ (resp. $\hat{n}''_{i} = n_{i}, \hat{k}''_{i} =
k_{i}$). Therefore, the ``expected jump'' of $V_{n_{i}, k_{i}}$ at
$\tau_{1}$ is dominated by the right and left ``expected drop'' of
$V_{\hat{n}'_{i}, \hat{k}'_{i}}$ and $V_{\hat{n}''_{i},
  \hat{k}''_{i}}$, respectively. 

Let $(n_{R}, k_{R})$ be the index for which $V_{n_{R}, k_{R}}(\mathbf{E})$ has
biggest ``right drop'' at $\tau_{1}$, i.e., 
\begin{eqnarray*}
 & & V_{n, k}(\mathbf{E}) \mathbb{P}_{\mathbf{E}}[ \mathcal{C}_{k+n}(
\tau_{1})] - \mathbb{E}_{\mathbf{E}}[ V_{n,k}( \mathbf{E}_{\tau_{1}^{+}})
\mathbf{1}_{\mathcal{C}_{k+n}(\tau_{1})}] \\
&\leq & V_{n_{R}, k_{R}}(\mathbf{E}) \mathbb{P}_{\mathbf{E}}[ \mathcal{C}_{k_{R}+n_{R}}(
\tau_{1})] - \mathbb{E}_{\mathbf{E}}[ V_{n_{R},k_{R}}(
\mathbf{E}_{\tau_{1}^{+}}) \mathbf{1}_{\mathcal{C}_{k_{R} +n_{R}}(\tau_{1})}]
\end{eqnarray*}
for all $n, k$, and $(n_{L}, k_{L})$ be the index for which $V_{n_{L}, k_{L}}( \mathbf{E})$ 
has the biggest ``left drop'', i.e.
\begin{eqnarray*}
 & & V_{n, k}(\mathbf{E}) \mathbb{P}_{\mathbf{E}}[ \mathcal{C}_{k}(
\tau_{1})] - \mathbb{E}_{\mathbf{E}}[ V_{n,k}( \mathbf{E}_{\tau_{1}^{+}})
\mathbf{1}_{\mathcal{C}_{k}(\tau_{1})}] \\
&\leq & V_{n_{L}, k_{L}}(\mathbf{E}) \mathbb{P}_{\mathbf{E}}[ \mathcal{C}_{k_{L}}(
\tau_{1})] - \mathbb{E}_{\mathbf{E}}[ V_{n_{L},k_{L}}(
\mathbf{E}_{\tau_{1}^{+}}) \mathbf{1}_{\mathcal{C}_{k_{L}}(\tau_{1})}] 
\end{eqnarray*}
for all $n, k$. Both terms should be positive when $M'$ is
sufficiently large because of the argument about $V_{\tilde{n},
  \tilde{k}}$ in {\bf Step 1}. 

Since there is only $\frac{1}{2}N(N+1)$ $V_{n,k}$ s, we have $m \leq
\frac{1}{2}N(N+1)$. Hence
\begin{eqnarray*}
 & &  \sum_{i = 1}^{m} \{\mathbb{E}_{\mathbf{E}}[ V_{n_{i}, k_{i}}(
 \mathbf{E}_{\tau_{1}^{+}})] - V_{n_{i},
   k_{i}}(\mathbf{E}) \}\\
&=& \sum_{i = 1}^{m} \{\mathbb{E}_{\mathbf{E}}[ V_{n_{i}, k_{i}}(
\mathbf{E}_{\tau_{1}^{+}}) \mathbf{1}_{\mathcal{C}_{k_{i}+n_{i}}(\tau_{1})}] - V_{n_{i}, k_{i}}(
\mathbf{E}) \mathbb{P}_{\mathbf{E}}[ \mathcal{C}_{k_{i} + n_{i}}( \tau_{1})] \\
&&+ \mathbb{E}_{\mathbf{E}}[ V_{n_{i}, k_{i}}( \mathbf{E}_{\tau_{1}^{+}})
\mathbf{1}_{\mathcal{C}_{k_{i}}(\tau_{1})}] - V_{n_{i}, k_{i}}(
\mathbf{E}) \mathbb{P}_{\mathbf{E}}[ \mathcal{C}_{k_{i}}( \tau_{1})] \}\\
&\leq& \frac{1}{4N^{2}} \sum_{i = 1}^{m} \left \{
  \max\{V_{\hat{n}'_{i}, \hat{k}'_{i}}(
\mathbf{E}) \mathbb{P}_{\mathbf{E}}[ \mathcal{C}_{\hat{k}'_{i} + \hat{n}'_{i}}(
\tau_{1})] - \mathbb{E}_{\mathbf{E}}[V_{\hat{n}'_{i}, \hat{k}'_{i}}(
  \mathbf{E}_{\tau_{1}^{+}}) \mathbf{1}_{
\mathcal{C}_{\hat{k}'_{i}+\hat{n}'_{i}}(\tau_{1})}] , 0 \} \right .\\
& &+ \left . \max \{ V_{\hat{n}''_{i}, \hat{k}''_{i}}(
\mathbf{E}) \mathbb{P}_{\mathbf{E}}[ \mathcal{C}_{\hat{k}''_{i}}(
    \tau_{1}) ]- \mathbb{E}_{\mathbf{E}}[ V_{\hat{n}''_{i}, \hat{k}''_{i}}( \mathbf{E}_{\tau_{1}^{+}})
\mathbf{1}_{\mathcal{C}_{\hat{k}''_{i}}(\tau_{1})} ]   ,
0\} \right \}
\\
&\leq& \frac{1}{4} \{ V_{n_{R}, k_{R}}(\mathbf{E}) \mathbb{P}_{\mathbf{E}}[ \mathcal{C}_{k_{R}+n_{R}}(
\tau_{1})] - \mathbb{E}_{\mathbf{E}}[ V_{n_{R}, k_{R}}( \mathbf{E}_{\tau_{1}^{+}})
\mathbf{1}_{\mathcal{C}_{k_{R}+n_{R}}(\tau_{1})}] \\
&& + V_{n_{L}, k_{L}}(\mathbf{E}) \mathbb{P}_{\mathbf{E}}[ \mathcal{C}_{k_{L}}(
\tau_{1})] - \mathbb{E}_{\mathbf{E}}[ V_{n_{L},k_{L}}(
\mathbf{E}_{\tau_{1}^{+}}) \mathbf{1}_{\mathcal{C}_{k_{L}}(\tau_{1})}]  \} \,.
\end{eqnarray*}

{\bf Step 3.} Then we can finalize the entire proof. We have
\begin{eqnarray*}
 & & \mathbb{E}_{\mathbf{E}}( V( \mathbf{E}_{\tau_{1}^{+}}) ) - V(
 \mathbf{E}) \\
&=& \sum_{n, k} \{\mathbb{E}_{\mathbf{E}}[ V_{n, k}(
\mathbf{E}_{\tau_{1}^{+}}) \mathbf{1}_{\mathcal{C}_{k+n}(\tau_{1})}] - V_{n, k}(
\mathbf{E}) \mathbb{P}_{\mathbf{E}}[ \mathcal{C}_{k + n}( \tau_{1})] \\
&&+ \mathbb{E}_{\mathbf{E}}[ V_{n, k}( \mathbf{E}_{\tau_{1}^{+}})
\mathbf{1}_{\mathcal{C}_{k}(\tau_{1})}] - V_{n, k}(
\mathbf{E}) \mathbb{P}_{\mathbf{E}}[ \mathcal{C}_{k}( \tau_{1})] \}
   \quad \mbox{(Law of total expectation)}\\
&=& \sum_{i = 1}^{m} \{\mathbb{E}_{\mathbf{E}}[ V_{n_{i}, k_{i}}(
\mathbf{E}_{\tau_{1}^{+}}) \mathbf{1}_{\mathcal{C}_{k_{i}+n_{i}}(\tau_{1})}] - V_{n_{i}, k_{i}}(
\mathbf{E}) \mathbb{P}_{\mathbf{E}}[ \mathcal{C}_{k_{i} + n_{i}}( \tau_{1})] \\
&&+ \mathbb{E}_{\mathbf{E}}[ V_{n_{i}, k_{i}}( \mathbf{E}_{\tau_{1}^{+}})
\mathbf{1}_{\mathcal{C}_{k_{i}}(\tau_{1})}] - V_{n_{i}, k_{i}}(
\mathbf{E}) \mathbb{P}_{\mathbf{E}}[ \mathcal{C}_{k_{i}}( \tau_{1})]
\} \\
&& + \sum_{k,n \, \, \,V_{n,k} < M} \{ \mathbb{E}_{\mathbf{E}}[ V_{n, k}(
 \mathbf{E}_{\tau_{1}^{+}})] - V_{n, k}(\mathbf{E})
 \} \qquad \mbox{ (By Lemma 4.5)}\\
&\leq& \frac{3}{4}  \{  \mathbb{E}_{\mathbf{E}}[ V_{n_{R},k_{R}}(
\mathbf{E}_{\tau_{1}^{+}}) \mathbf{1}_{\mathcal{C}_{n_{R} +
    k_{R}}(\tau_{1})}] -V_{n_{R}, k_{R}}(\mathbf{E})
\mathbb{P}_{\mathbf{E}}[ \mathcal{C}_{n_{R}+k_{R}}( \tau_{1})] \\
&& +\mathbb{E}_{\mathbf{E}}[ V_{n_{L}, k_{L}}( \mathbf{E}_{\tau_{1}^{+}})
\mathbf{1}_{\mathcal{C}_{k_{L}}(\tau_{1})}] - V_{n_{L}, k_{L}}(\mathbf{E}) \mathbb{P}_{\mathbf{E}}[ \mathcal{C}_{k_{L}}(
\tau_{1})]  \} \\
&&+ \sum_{k,n \, \, \,V_{n,k} < M} \{ \mathbb{E}_{\mathbf{E}}[ V_{n, k}(
 \mathbf{E}_{\tau_{1}^{+}}) ] - V_{n, k}(\mathbf{E})
 \}  \qquad \mbox{(By {\bf Step 2})}\\
&\leq& \frac{3}{4}(I_{3} + I_{4}) + I_{5}  \,,
\end{eqnarray*}
where
$$
  I_{3} = \mathbb{E}_{\mathbb{E}}[ V_{n_{R},k_{R}}(
\mathbf{E}_{\tau_{1}^{+}}) \mathbf{1}_{\mathcal{C}_{n_{R} +
    k_{R}}(\tau_{1})}] -V_{n_{R}, k_{R}}(\mathbf{E})
\mathbb{P}_{\mathbf{E}}[ \mathcal{C}_{n_{R}+k_{R}}( \tau_{1})]  \,,
$$
$$
  I_{4} = \mathbb{E}_{\mathbf{E}}[ V_{n_{L}, k_{L}}( \mathbf{E}_{\tau_{1}^{+}})
\mathbf{1}_{\mathcal{C}_{k_{L}}(\tau_{1})}] - V_{n_{L}, k_{L}}(\mathbf{E}) \mathbb{P}_{\mathbf{E}}[ \mathcal{C}_{k_{L}}(
\tau_{1})]   \,,
$$
and
$$
  I_{5} =  \sum_{k,n \, \, \,V_{n,k} < M} \{ \mathbb{E}_{\mathbf{E}}[ V_{n, k}(
 \mathbf{E}_{\tau_{1}^{+}}) ] - V_{n, k}(\mathbf{E})
 \}  \,.
$$
It follows from Lemma \ref{cor1} that
\begin{eqnarray*}
I_{5}&=& \sum_{k,n \, \, \,V_{n, k} < M} \{\mathbb{E}_{\mathbf{E}}[ V_{n, k}( \mathbf{E}_{\tau_{1}^{+}})
\mathbf{1}_{\mathcal{C}_{k+n}(\tau_{1})}] - V_{n, k}(
\mathbf{E}) \mathbb{P}_{\mathbf{E}}[ \mathcal{C}_{k + n}( \tau_{1})] \\
&&+ \mathbb{E}_{\mathbf{E}}[ V_{n, k}( \mathbf{E}_{\tau_{1}^{+}}),
\mathbf{1}_{\mathcal{C}_{k}(\tau_{1})}] - V_{n, k}(
\mathbf{E}) \mathbb{P}_{\mathbf{E}}[ \mathcal{C}_{k}( \tau_{1})] \}\\
&\leq&\sum_{k,n \, \, \,V_{n,k} < M} \{\mathbb{E}_{\mathbf{E}}[ V_{n, k}(
\mathbf{E}_{\tau_{1}^{+}}) \,|\, \mathcal{C}_{k+n}(\tau_{1})] \mathbb{P}_{\mathbf{E}}[ \mathcal{C}_{k+n}(\tau_{1})]  +
\mathbb{E}_{\mathbf{E}}[ V_{n, k}( \mathbf{E}_{\tau_{1}^{+}}) \,|\ 
\mathcal{C}_{k}(\tau_{1})]
\mathbb{P}_{\mathbf{E}}[\mathcal{C}_{k}(\tau_{1})]\} \\ 
&\leq&  \sum_{k,n \, \, \,V_{n,k} < M} \{\mathbb{E}_{\mathbf{E}}[ V_{n, k}(
\mathbf{E}_{\tau_{1}^{+}}) \,|\, \mathcal{C}_{k+n}(\tau_{1})]   +
\mathbb{E}_{\mathbf{E}}[ V_{n, k}( \mathbf{E}_{\tau_{1}^{+}}) \,|\ 
\mathcal{C}_{k}(\tau_{1})] \} \cdot \frac{2K}{\mathcal{R}}\\
&\leq&N(N+1)CMK\cdot \frac{1}{\mathcal{R}} \,.
\end{eqnarray*}
In addition, by the definition of $n_{R}, k_{R}$ and $n_{L}, k_{L}$,
we have
\begin{eqnarray*}
 & & V_{\tilde{n}, \tilde{k}}(\mathbf{E}) \mathbb{P}_{\mathbf{E}}[ \mathcal{C}_{\tilde{k}+\tilde{n}}(
\tau_{1})] - \mathbb{E}_{\mathbf{E}}[ V_{\tilde{n},\tilde{k}}( \mathbf{E}_{\tau_{1}^{+}})
\mathbf{1}_{\mathcal{C}_{\tilde{k}+\tilde{n}}(\tau_{1})}] \\
&\leq & V_{n_{R}, k_{R}}(\mathbf{E}) \mathbb{P}_{\mathbf{E}}[ \mathcal{C}_{k_{R}+n_{R}}(
\tau_{1})] - \mathbb{E}_{\mathbf{E}}[ V_{n_{R},k_{R}}(
\mathbf{E}_{\tau_{1}^{+}}) \mathbf{1}_{\mathcal{C}_{k_{R} +n_{R}}(\tau_{1})}]
\end{eqnarray*}
and
\begin{eqnarray*}
 & & V_{\tilde{n}, \tilde{k}}(\mathbf{E}) \mathbb{P}_{\mathbf{E}}[ \mathcal{C}_{\tilde{k}}(
\tau_{1})] - \mathbb{E}_{\mathbf{E}}[ V_{\tilde{n},\tilde{k}}( \mathbf{E}_{\tau_{1}^{+}})
\mathbf{1}_{\mathcal{C}_{\tilde{k}}(\tau_{1})}] \\
&\leq & V_{n_{L}, k_{L}}(\mathbf{E}) \mathbb{P}_{\mathbf{E}}[ \mathcal{C}_{k_{L}}(
\tau_{1})] - \mathbb{E}_{\mathbf{E}}[ V_{n_{L},k_{L}}(
\mathbf{E}_{\tau_{1}^{+}}) \mathbf{1}_{\mathcal{C}_{k_{L}}(\tau_{1})}]  \,,
\end{eqnarray*}
where $\tilde{n}$ and $\tilde{k}$ are from {\bf Step 1}.

By inequality \eqref{controlofii} in {\bf Step 1}, we have
\begin{eqnarray*}
  I_{3} + I_{4} &\leq& (\mathbb{E}_{\mathbf{E}}[ V_{\tilde{n},\tilde{k}}( \mathbf{E}_{\tau_{1}^{+}})
\mathbf{1}_{\mathcal{C}_{\tilde{k}+\tilde{n}}(\tau_{1})}]  - V_{\tilde{n}, \tilde{k}}(\mathbf{E}) \mathbb{P}_{\mathbf{E}}[ \mathcal{C}_{\tilde{k}+\tilde{n}}(
\tau_{1})] )\\
&& + (\mathbb{E}_{\mathbf{E}}[ V_{\tilde{n},\tilde{k}}( \mathbf{E}_{\tau_{1}^{+}})
\mathbf{1}_{\mathcal{C}_{\tilde{k}}(\tau_{1})}]  - V_{\tilde{n}, \tilde{k}}(\mathbf{E}) \mathbb{P}_{\mathbf{E}}[ \mathcal{C}_{\tilde{k}}(
\tau_{1})]  )\\
&= & \mathbb{E}_{\mathbf{E}}[ V_{\tilde{n}, \tilde{k}}(
\mathbf{E}_{\tau_{1}^{+}})] - V_{\tilde{n}, \tilde{k}}( \mathbf{E}) \\
&\leq& -\frac{1}{\mathcal{R}} V_{\tilde{n},
    \tilde{k}}^{\alpha}( \mathbf{E}) \leq - \frac{1}{\mathcal{R}}
  \cdot \left ( \frac{1}{N(N+1)} \right )^{\alpha} V^{\alpha}(
  \mathbf{E}) \,.
\end{eqnarray*}

One can then further choose $M' > M$ such that
$$
  \frac{1}{4\mathcal{R}} \cdot (\frac{1}{N(N+1)})^{\alpha}
  M'^{\alpha} > I_{5} \,.
$$
Therefore, for every $\mathbf{E}$ such that $V(E) > M'$, we
have
$$
  \mathbb{E}_{\mathbf{E}}( V( \mathbf{E}_{\tau_{1}^{+}})) - V(
 \mathbf{E}) \leq -\frac{1}{\mathcal{R}} \cdot \frac{1}{2}
 (\frac{1}{N(N+1)})^{\alpha} V^{\alpha}( \mathbf{E}) :=
 -\frac{C_{1}}{\mathcal{R}} V^{\alpha}( \mathbf{E})\,.
$$
This completes the proof.

\end{proof}

{\bf \noindent Proof of Theorem \ref{timestep}.}

Let $0 = \tau_{0} < \tau_{1} < \tau_{2} < \cdots $ be the times of clock rings. Let
$B_{0} = \{ \mathbf{E} \,|\, V(\mathbf{E}) > M'\}$ where $M'$ is the
constant in Lemma \ref{stoppingtime}. By the Markov property and Lemma \ref{stoppingtime}, for
any $\mathbf{E} \in \mathbb{R}^{N}_{+}$ we
have 
\begin{equation}
\label{eq1}
  \mathbb{E}_{\mathbf{E}}[ V( \mathbf{E}_{\tau_{n+1}^{+}}) \,|\,
  \mathbf{E}_{\tau_{n}^{+}}] = \mathbb{E}_{\mathbf{E}_{\tau_{n}^{+}}}[ V( \mathbf{E}_{\tau_{1}^{+}}) ]
\leq V(\mathbf{E}_{\tau_{n}^{+}}) -
  \frac{C_{1}}{\mathcal{R}} V(\mathbf{E}_{\tau_{n}^{+}})^{\alpha} 
\end{equation}
if $\mathbf{E}_{\tau_{n}^{+}} \in B_{0}$. Otherwise, for each
$\mathbf{E}_{\tau_{n}^{+}} \notin B_{0}$, by the Markov property we have the uniform bound 
\begin{eqnarray}
\label{eq2}
 &&\mathbb{E}_{\mathbf{E}}[ V( \mathbf{E}_{\tau_{n+1}^{+}}) \,| \, \mathbf{E}_{\tau_{n}^{+}} ] \\\nonumber
 & = & \sum_{m = 1}^{N}\sum_{k = 1}^{N - m + 1} \left \{\mathbb{E}_{\mathbf{E}_{\tau_{n}^{+}}}[ V_{m, k}(
  \mathbf{E}_{\tau_{1}^{+}}) \,|\, \mathcal{C}_{k}(\tau_{1})]
  \mathbb{P}_{\mathbf{E}_{\tau_{n}^{+}}}[ \mathcal{C}_{k}(\tau_{1})]
       \right .
  \\\nonumber
&&+
  \left .\mathbb{E}_{\mathbf{E}_{\tau_{n}^{+}}}[ V_{m, k}(
  \mathbf{E}_{\tau_{1}^{+}}) \,|\, \mathcal{C}_{k + m}(\tau_{1})]
   \mathbb{P}_{\mathbf{E}_{\tau_{n}^{+}}}[ \mathcal{C}_{k +
   m}(\tau_{1})] \right \}
  \\\nonumber
&\leq& \sum_{m = 1}^{N}\sum_{k = 1}^{N - m + 1} C V_{m, k}( \mathbf{E}_{\tau_{n}^{+}}) (
       \mathbb{P}_{\mathbf{E}_{\tau_{n}^{+}}}[
       \mathcal{C}_{k}(\tau_{1})] + \mathbb{P}_{\mathbf{E}_{\tau_{n}^{+}}}[ \mathcal{C}_{k +
   m}(\tau_{1})]  ) \quad \mbox{ (By Lemma 4.3 )}\\\nonumber
&\leq& C \sum_{m = 1}^{N}\sum_{k = 1}^{N - m + 1} V_{m, k}( \mathbf{E}_{\tau_{n}^{+}}) \leq  C M' \,,
\end{eqnarray}
where $C$ is the constant in Lemma 4.2. 

Let $S = \inf \{ n \,|\, \tau_{n} > h \}$, and define $\hat{\tau}_{n} =
\min \{ \tau_{n}, \tau_{S-1 } \}$. Then 
$$
  P^{h}V( \mathbf{E}) = \lim_{n \rightarrow \infty} \mathbb{E}_{\mathbf{E}}[ V(
  \mathbf{E}_{\hat{\tau}_{n}^{+}}) \mathbf{1}_{\{S\leq n+1 \}}] \leq \limsup_{n\rightarrow \infty} \mathbb{E}_{\mathbf{E}}[ V(
  \mathbf{E}_{\hat{\tau}_{n}^{+}}) ]  \,.
$$ 
We will prove a uniform bound for $\mathbb{E}_{\mathbf{E}}[ V(
  \mathbf{E}_{\hat{\tau}_{n}^{+}}) ] $. Equation \eqref{eq1} implies the
  expectation of $V$ drops when starting from $B_{0}$. Equation
  \eqref{eq2} means the expectation of $V$ can grow with $C M'$ at
  most. By assuming the worse of \eqref{eq1}
  and \eqref{eq2}, we have
$$
  \mathbb{E}_{\mathbf{E}}[ V(
  \mathbf{E}_{\hat{\tau}_{n+1}^{+}})  \,|\,
  \tau_{n+1} \leq h]  \leq \mathbb{E}_{\mathbf{E}}[ V(
  \mathbf{E}_{\hat{\tau}_{n}^{+}})  \,|\,
  \tau_{n+1} \leq h] +  C M'
$$
for every $n \geq 0$. Notice that for given $\mathbf{E}_{\tau_{n}^{+}}$,
$\mathbf{E}_{\tau_{n+1}^{+} }$ is independent of $\tau_{n+1} -
\tau_{n}$. Therefore conditioning on $\tau_{n+1 } \leq h$ does not
affect the bounds in \eqref{eq1} and \eqref{eq2}. Since $\mathcal{R}
\leq (N + 1)K$, for each $\mathbf{E}_{\tau_{n}^{+}}$, we have
$$
  \mathbb{P}[ \tau_{n+1} \leq h | \mathbf{E}_{\tau_{n}^{+}}, \tau_{n}
  \leq h] \leq (1 - e^{-hNK}) \,.
$$
Therefore, inductively we have
$$
  \mathbb{P}_{\mathbf{E}}[ \tau_{n+1} \leq h ] \leq (1 -
  e^{-hNK}) ^{n+1}
$$
for all $n \geq 0$. This implies
\begin{eqnarray}
\label{eq4} 
& \ & \mathbb{E}_{\mathbf{E}}[ V( \mathbf{E}_{\hat{\tau}_{n+1}^+})] \\\nonumber
& = & \mathbb{E}_{\mathbf{E}}[ V( \mathbf{E}_{\hat{\tau}_{n+1}^+}) \,|\, \tau_{n+1}>h] 
\cdot \mathbb P_{\mathbf{E}}[\tau_{n+1}>h] + 
\mathbb{E}_{\mathbf{E}}[ V( \mathbf{E}_{\hat{\tau}_{n+1}^+}) \,|\, \tau_{n+1}\leq h] 
\cdot \mathbb P_{\mathbf{E}}[\tau_{n+1}\leq h]\\\nonumber
& \le & \mathbb{E}_{\mathbf{E}}[ V( \mathbf{E}_{\hat{\tau}_{n}^+}) \,|\, \tau_{n+1}>h] 
\cdot \mathbb P_{\mathbf{E}}[\tau_{n+1}>h] \\\nonumber
& & \hskip 1cm + 
\left(\mathbb{E}_{\mathbf{E}}[ V( \mathbf{E}_{\hat{\tau}_{n}^+}) \,|\,
  \tau_{n+1}\leq h] +
C M'\right) \cdot \mathbb P_{\mathbf{E}}[\tau_{n+1} \leq h]\\\nonumber
& \le & \mathbb{E}_{\mathbf{E}}[ V( \mathbf{E}_{\hat{\tau}_n^+})] + 
C M' (1-e^{-hNK})^{n+1}\ .
\end{eqnarray}
Adding up from $n = 1$ to $\infty$, this gives
$$
  P^{h}V( \mathbf{E}) \leq \mathbb{E}_{\mathbf{E}}[V( \mathbf{E}_{\hat{\tau}_{1}^+})]  + \frac{C M' }{e^{-hNK}} \,.
$$

Let $h>0$ be small enough so that 
$$
 \mathbb{P}_{\mathbf{E}}[ \tau_{1} \leq h  ] = 1 - e^{-h
   \mathcal{R}} > \frac{h}{2} \mathcal{R}\ . 
  $$
This is the only condition we impose on $h$. There exists such an $h$
independently of $\mathbf{E}$ because of the bound $\mathcal{R} \leq NK$. 

Notice that the energy exchange at $\tau_{1}$ is independent of $\tau_{1}$. We have, by Lemma \ref{stoppingtime}, 
\begin{eqnarray*} 
\mathbb{E}_{\mathbf{E}}[V( \mathbf{E}_{\hat{\tau}_{1}^+}) ] & = &
\mathbb{E}_{\mathbf{E}}[V( \mathbf{E}_{\tau_{1}^+}) \,|\, \tau_1 \leq h] \cdot \mathbb P_{\mathbf{E}}[\tau_1 \leq h] + 
V(\mathbf{E}) \cdot \mathbb P_{\mathbf{E}}[\tau_1 > h]\\
& \le & \left(V(\mathbf{E}) - \frac{C_{1}}{\mathcal{R}} V(\mathbf{E})^\alpha \right) \cdot \mathbb P_{\mathbf{E}}[\tau_1 \leq h]
 + V(\mathbf{E}) \cdot \mathbb P_{\mathbf{E}}[\tau_1 > h]\\
& \le & V(\mathbf{E}) - C_{1} \frac{h}{2} V(\mathbf{E})^\alpha\ .
\end{eqnarray*}
This gives
$$
P^{h}V( \mathbf{E}) \leq V(\mathbf{E}) - C_{1}\frac{h}{2} V(\mathbf{E})^\alpha
+ C M' e^{h N K}
$$
for any $\mathbf{E} \in B_{0}$.

To complete the proof of Theorem \ref{timestep}, it suffices to replace $M'$ by 
a large enough number $M_{0} > 1$ so that for $\mathbf E \in \{V( \mathbf{E})>M_{0}\}$, the constant 
$C M' e^{hNK}$ is  
absorbed into $c_{0} V(\mathbf E)^\alpha$ for $c_{0} :=C_{1 } \frac{h}{4}$. 

\qed

The same calculation in the proof of Theorem \ref{timestep} also yields
\begin{lem}
\label{cor2}
$$
  \sup_{\{ \mathbf{E}\,:\, V(\mathbf{E}) \leq \hat{M} \}} P^{h}V(
  \mathbf{E}) \leq \hat{M} +  CM' e^{hNK} < \infty \,.
$$
for any $\hat{M} > 0$. 
\end{lem}
\begin{proof}
The calculation in \eqref{eq4} gives 
$$
  \mathbb{E}_{\mathbf{E}}[ V( \mathbf{E}_{\hat{\tau}_{n+1}^{+}})] \leq \mathbb{E}_{\mathbf{E}}[V(
  \mathbf{E}_{\hat{\tau}_{n}^{+}}) ] + C M'(1 - e^{-hNK})^{n+1} \,.
$$
Adding up from $n = 0$ to $\infty$, this gives
$$
  P^{h}V( \mathbf{E}) \leq V( \mathbf{E}) + CM'e^{hNK} \leq
  \hat{M} +  CM'e^{hNK}  \,.
$$
\end{proof}

\section{Excursion time for the induced chain}\label{high}

For sufficiently small given parameters $h > 0$ and $\eta > 0$, let $M_{0}$ be the constant defined in Theorem \ref{timestep} and $B:=
\{ \mathbf{E} \,|\, V(\mathbf{E}) > M_{0} \}$. Define $G =
\mathbb{R}^{N}_{+}\setminus B$. The previous section together with
Theorem 3.11 gives bounds on
the excursion time in $B$. As introduced in Section 3.2, now one
should work on the $G$-induced chain. $G$ is not a compact set as the
energy at each site can be arbitrarily high. A common way to show
the tightness of a Markov process on non-compact state space is to
construct a Lyapunov function, as we do in this section for the
$G$-induced chain.

Let $h > 0$ be the given size of a time step (defined in Theorem
\ref{timestep}). We consider the time-$h$
sampling chain $\{ \mathbf{E}_{nh}\}_{n =
  0}^{\infty}$ of $\mathbf{E}_{t}$. For the sake of simplicity, we use
the notation $\mathbf{E}_{n}$ to represent $\mathbf{E}_{nh}$ when it
does not lead to a confusion.

Let $0 = T_{0} <T_{1} < T_{2} < \cdots$ be discrete stopping times such that
$$
  T_{1} = \inf\{k > 0 \,|\, \mathbf{E}_{k} \in G \}
$$
and
$$
  T_{n+1} = \inf\{ k > T_{n} \,|\, \mathbf{E}_{k} \in G \}
$$
for $n = 1, 2, \cdots$. We define
$$
  \hat{\mathbf{E}}_{n} = \mathbf{E}_{T_{n}}
$$
for $n = 1, 2, \cdots$ as the $G$-induced chain. It is easy to see
that $\hat{\mathbf{E}}_{n}$ is also a Markov chain. We denote the
transition kernel of $\hat{\mathbf{E}}_{n}$ by $\hat{P}$.

As in \cite{li2016polynomial}, a natural Lyapunov function is the total energy in
the system. Let
$$
  W( \mathbf{E}) = \sum_{ n = 1}^{N} E_{n} \,.
$$
The main theorem of this section reads:

\begin{thm}
\label{lyapunovinduced}
There exist constants $M_{1} > 1$ and $\delta > 0$, such that 
$$
  \hat{P} W( \mathbf{E}) \leq (1 - \delta)W( \mathbf{E}) 
$$
for every $\mathbf{E} \in G$ with $W(\mathbf{E}) > M_{1}$. 
\end{thm}

To prove Theorem \ref{lyapunovinduced}, the first task is to bound the length of the time step for
$\hat{\mathbf{E}}_{n}$. Such estimate follows from Theorem
\ref{timestep} and Theorem \ref{polyfirstpassage} immediately.

\begin{pro}
\label{lem61}
There exists a constant $C_{5}$ such that
$$
   \mathbb{E}[ (T_{n+1} - T_{n})^{\hat{\alpha}} \,| \, \mathbf{\hat{E}}_{n}] \leq
   C_{5} \hat{V}(\mathbf{\hat{E}}_{n}) 
$$
for any $\mathbf{\hat{E}}_{n}$ and any $n \geq 0$, where $\hat{\alpha} = (
1- \alpha )^{-1} = 2 - 2 \eta$ for the constant $\alpha$ given in Theorem
\ref{timestep}, and $\hat{V}( \mathbf{E}) = \max \{ V( \mathbf{E}), 1 \}$. In particular, if $n
\geq  1$, then 
$$
  \mathbb{E}[ (T_{n+1} - T_{n})^{\hat{\alpha}} \,| \, \mathbf{\hat{E}}_{n}] \leq
  C_{5} M_{0} \,,
$$
where $M_{0}$ is as in Theorem \ref{timestep}. 
\end{pro}
\begin{proof}
By the definition of $\mathbf{\hat{E}}_{n}$, $\mathbf{\hat{E}}_{n} =
\mathbf{E}_{T_{n}}$ is the energy configuration at the stopping time
$T_{n}$. Notice that
$$
  \mathbb{E}[ (T_{n+1} - T_{n})^{\hat{\alpha}} \, |\, \mathbf{\hat{E}}_{n}] \leq
  2 \cdot \mathbb{E}_{\mathbf{E}_{T_{n}}} \left [ \sum_{k =
      0}^{\tau_{G} - 1} (k+1)^{\hat{\alpha} - 1}\right] 
$$
for the constant $\alpha$ we use. The proposition follows immediately
by applying Theorem \ref{polyfirstpassage} to $\mathbf{E}_{n}$. Since
$M_{0} > 1$, let
$$
  \hat{V}( \mathbf{E}) = \max \{ 1, V( \mathbf{E}) \} \,.
$$
Then since $\hat{V} \leq V + 1$, it
follows from Theorem \ref{timestep} and Lemma \ref{cor2} that
$$
  \hat{P}\hat{V}( \mathbf{E}) - \hat{V}(\mathbf{E}) \leq -c_{0} \hat{V}^{\alpha}(
  \mathbf{E}) + ( 1 + M_{0} + C M' e^{hNK}) \mathbf{1}_{G} \,.
$$

The proposition then follows from Theorem \ref{polyfirstpassage}.
\end{proof}

The following definitions
regarding the energy flux in the system are necessary for the proof of
Theorem \ref{lyapunovinduced}. Let $t_{1},
t_{2}, \cdots$ be the times at which either clock $1$ or clock $N+1$
rings. The energy in-flow and out-flow on $[0, T)$ are denoted by 
$$
  F_{I}([0, T)) = \sum_{0 \leq t_{i} < T} (W(\mathbf{E}_{t_{i}^{+}}) - W(\mathbf{E}_{t_{i}}) )^{+}
$$
and 
$$
  F_{O}([0, T)) = \sum_{0 \leq t_{i} < T} ( W(\mathbf{E}_{t_{i}}) -
  W(\mathbf{E}_{t_{i}^{+}}) )^{+} \,,
$$
respectively. Next we need to estimate the energy flux with respect to
$\hat{\mathbf{E}}$. 

\begin{lem}
\label{inflow1}
There exists a constant $C_{2}$ such that
$$
  \mathbb{E}_{\mathbf{E}}[ F_{I}( [0, T)) ] \leq C_{2} T
$$
for any $\mathbf{E} \in \mathbb{R}^{N}_{+}$ and any $T > 0$. In addition
$$
  X_{n} := F_{I}( [0, nh)) - C_{2}hn
$$
is a supermartingale relative to $\mathcal{F}_{n}$, where
$\mathcal{F}_{n}$ is the $\sigma$ field generated by
$\{\mathbf{E}_{0}, \cdots,\mathbf{E}_{n}\}$. 
\end{lem}
\begin{proof}
It is easy to see that for any $\mathbf{E}(t) \in \mathbb{R}^{N}_{+}$ we have
\begin{eqnarray*}
&&\mathbb{E}_{\mathbf{E}(t)}[F_{I}([t, t+ \mathrm{d}t ))]\\
 & \leq  & \sqrt{T_{L}} \left \{\int_{0}^{\infty}
\int_{0}^{1} \frac{1}{T_{L}} p(x + E_{1}(t)) e^{-x/T_{L}} \mathrm{d}x
\mathrm{d}p -  E_{1} \right \}^{+} \mathrm{d}t\\
&&+ \sqrt{T_{R}} \left \{\int_{0}^{\infty}
\int_{0}^{1} \frac{1}{T_{R}} p(x + E_{N}(t)) e^{-x/T_{R}} \mathrm{d}x
\mathrm{d}p -  E_{N} \right \}^{+} \mathrm{d}t\\
&\leq& \sqrt{T_{L}} \left \{\int_{0}^{\infty}
\int_{0}^{1} \frac{1}{T_{L}} p x e^{-x/T_{L}} \mathrm{d}x
\mathrm{d}p  \right \} \mathrm{d}t + \sqrt{T_{R}} \left \{\int_{0}^{\infty}
\int_{0}^{1} \frac{1}{T_{R}} p x  e^{-x/T_{R}} \mathrm{d}x
\mathrm{d}p  \right \} \mathrm{d}t\\
&=& \frac{1}{2}(T_{L}^{3/2} + T_{R}^{3/2}) \mathrm{d}t := C_{2} \mathrm{d}t\,.
\end{eqnarray*}

This estimate is independent of $\mathbf{E}(t)$. Therefore,
$$
  \mathbb{E}_{\mathbf{E}}[ F_{I}( [0, T))] \leq C_{2} T
$$
for any $\mathbf{E}$ and $T$. In addition, 
$$
  \mathbb{E}[X_{n+1} - X_{n} \,|\, \mathcal{F}_{n}] = \mathbb{E}[
  F_{I}( [nh, (n+1)h)) \,|\, \mathcal{F}_{n} ] - C_{2} h \leq 0
$$
for any $\mathbf{E}_{n}$. This completes the proof.
\end{proof}

By the definition of $T_{1}$, $T_{1}$ is a stopping time relative to
$\mathcal{F}_{n}$. We have
\begin{pro}
\label{inflow2}
$$
  \mathbb{E}_{\mathbf{E}}[ F_{I}( [0, T_{1}))] \leq C_{2} h
  \mathbb{E}_{\mathbf{E}}[T_{1}] \,,
$$
where $C_{2}$ is the constant in Lemma \ref{inflow1}.
\end{pro}
\begin{proof}
By Proposition \ref{lem61}, $\mathbb{E}_{\mathbf{E}}[T_{1}] < \infty$ for any
$\mathbf{E} \in \mathbb{R}^{N}_{+}$. In addition, by Lemma \ref{inflow1},
$$
  \mathbb{E}_{\mathbf{E}}[|X_{n+1} - X_{n}| \,|\, \mathcal{F}_{n}] \leq 2 C_{2} h <
  \infty \,,
$$
where the super-martingale $X_{n}$ is defined in Lemma \ref{inflow1}. It then follows from the optional stopping theorem that
$$
  \mathbb{E}_{\mathbf{E}}[ X_{T_{1}} ] \leq \mathbb{E}_{\mathbf{E}}[ X_{0}] = 0 \,.
$$
Therefore 
$$
  \mathbb{E}_{\mathbf{E}}[ F_{I}( [0, T_{1}))] \leq C_{2} h
  \mathbb{E}_{\mathbf{E}}[T_{1}] \,.
$$
\end{proof}

The following estimate about the energy influx with respect to the
$G$-induced chain follows easily.

\begin{lem}
\label{inflow3}
There exist constants $C_{3} < \infty$ and $C_{3}' < \infty$ such that
$$
\mathbb{E}_{\mathbf{E}}[  F_{I}( [0 , T_{1})) ] \leq
C_{3} 
$$ 
if $\mathbf{E} \in G$ and 
$$
  \mathbb{E}_{\mathbf{E}}[  F_{I}( [0 , T_{1})) ] \leq
 C'_{3} \hat{V}(\mathbf{E}) \,.
$$
if $\mathbf{E} \notin G$.
\end{lem}
\begin{proof}
By Proposition \ref{lem61},
$$
  \mathbb{E}_{\mathbf{E}}[ T_{1}] \leq
  \mathbb{E}_{\mathbf{E}}[ (T_{1})^{\hat{\alpha}}]
  \leq C_{5} \hat{V}(\mathbf{E}) \,. 
$$

It then follows from Proposition \ref{inflow2} that
$$
  \mathbb{E}_{\mathbf{E}}[  F_{I}( [0, T_{1})) ] \leq
  C_{2}h  \cdot C_{5} \hat{V}(
  \mathbf{E})    \,.
$$

Let
$$
  C_{3} :=  C_{2}h  \cdot C_{5} M_{0} \,.
$$
If $\mathbf{E}_{0} \in G$, the lemma follows immediately. If
$\mathbf{E}_{0} \notin G$, by letting 
$$
  C_{3}' = C_{2}hC_{5}  \,,
$$
we will have 
$$
  \mathbb{E}_{\mathbf{E}}[  F_{I}( [0, T_{1}))  ] \leq
  C_{3}' \hat{V}( \mathbf{E}) \,.
$$
This completes the proof.

\end{proof}

The following lemma controls the out flow of the energy.

\begin{lem}
\label{outflow}
Assume $\mathbf{E} \in G$. There exist constants $\sigma, M'' > 0$, such that
$$
  \mathbb{E}_{\mathbf{E}}[ F_{O}( [0, T_{1}))] \geq \sigma W(
  \mathbf{E}) 
$$ 
whenever $W(\mathbf{E}) > M''$. 
\end{lem}
\begin{proof}
Since 
$$
  \mathbb{E}_{\mathbf{E}}[ F_{O}( [0, T_{1}))] \geq \mathbb{E}_{\mathbf{E}}[ F_{O}(
  [0, h))] \,, 
$$
it is sufficient to construct an event within the time interval $ [0 ,
h )$ such that whenever $\mathbf{E} \in G$, a certain proportion of energy can be dumped out of
the system.

Let $E_{n_{1}}$ be the site that holds the largest amount of
energy. Let $\mathcal{E}$ denote the following event
\begin{itemize}
  \item Clocks $n_{1}$, $n_{1} - 1$, $\cdots $, $2$, and $1$ ring in
    the time interval $[0, \frac{h}{n_{1}})$, $[ \frac{h}{n_{1}},
    \frac{2h}{n_{1}})$, $\cdots$, $[\frac{(n_{1} - 1)h}{n_{1}}, h)$,
    respectively.
\item At the $i$-th ring, $E_{n_{1} - i + 1}$ gives at least half of
  its energy to $E_{n_{1} - i}$ for $i = 1, \cdots, n_{1} - 1$.
\item At the $n_{1}$-th ring, $E_{1}$ dumps $1/3$ of its energy to the
  left heat bath.
\item Besides what described above, all other clocks do not ring
  during the time period $[0, h )$.
\end{itemize}
If $W( \mathbf{E}) \geq 3 N \cdot 2^{N-1} T_{L}$, we have
$E_{n_{1}}(0) 
\geq 3 \cdot 2^{N-1} \cdot T_{L}$ and $E_{1}(\frac{(n_{1} - 1)h}{n_{1}} 
) > 3 T_{L}$ conditioning with event $\mathcal{E}$. Note that all
clock rates are bounded above by $K$. In addition, since $E_{n_{1}}$
holds the largest amount of energy, right before the $i$-th ring (for
$i < n_{1}$) we
have $E_{n_{1}-i+1} \geq 2^{-(i-1)} E_{n_{1} - i}$. Hence the probability that
$E_{n_{1} - i + 1 } $ gives at least half of its energy to $E_{n_{1} -
i}$ is at least $2^{-i}/(1 + 2^{-(i-1)})$. Therefore, it is a
straightforward exercise to check that there exists a constant $c_{0}
> 0$ such that 
$$
  \mathbb{P}[ \mathcal{E}] \geq c_{0} > 0
$$
for every $\mathbf{E} \in G$. The proof is completed by letting
$$
  \sigma = \frac{c_{0}}{3} \cdot 2^{-(N-1)}
$$
and
$$
  M'' = 3 N \cdot 2^{N-1} T_{L} \,.
$$

\end{proof}

\begin{proof}
{\bf Proof of Theorem \ref{lyapunovinduced}.}

By the definition of $W( \mathbf{E})$, we have
$$
  \hat{P} W( \mathbf{E}) - W(\mathbf{E}) = \mathbb{E}_{ \mathbf{E}} [F_{I}([0,
  T_{1}))] - \mathbb{E}_{\mathbf{E}}[ F_{O}( [0, T_{1})) ]\,.
$$
Since $\mathbf{E} \in G$, by Lemma \ref{inflow3}, we have $\mathbb{E}_{\mathbf{E}}[F_{I}([0,
T_{1}) ) ]\leq C_{3}$. On the other hand, by Lemma \ref{outflow}, if
$W(E) > M''$, we have 
$$
  \mathbb{E}_{\mathbf{E}}[ F_{O}( [0, T_{1})) ] \geq \sigma W(
  \mathbf{E}) \,.
$$
Therefore, let $\delta = \frac{1}{2} \sigma$ and $M_{1} =
\max \{ 2C_{3}/\sigma , M'' \}$, we have 
$$
  \hat{P} W( \mathbf{E}) \leq (1 - \delta) W( \mathbf{E}) \,.
$$
This completes the proof.

\end{proof}

\section{Proof of the theorems}\label{pf}

\subsection{Existence of a uniform reference set $\mathfrak{C}$}

Define
$$
  \mathfrak{C} = \{ \mathbf{E} \,|\, V( \mathbf{E}) \leq M_{0}, W( \mathbf{E})
  \leq M_{1} \} \,.
$$

The aim of this subsection is to prove that $\mathfrak{C}$ is a uniform reference
set. This follows immediate from the theorem below.

\begin{thm}
\label{urs}
For any $ t > 0$, there exists a constant $\eta > 0$ such that 
$$
  P^{t}(\mathbf{E}, \cdot) > \eta U_{\mathfrak{C}}(\cdot) \quad \mbox{for all } \mathbf{E}
  \in \mathfrak{C} \,,
$$
where $U_{\mathfrak{C}}$ is the uniform probability measure on $\mathfrak{C}$.
\end{thm}
\begin{proof}

Let $e = \inf\{ \min(E_{1},
\cdots, E_{N}) \,|\, \mathbf{E} = (E_{1} , \cdots, E_{N}) \in
\mathfrak{C} \}$. Clearly $e > 0$. 

Note that $\mathfrak{C}$ is compact due to the condition involving
$W$. Then we cover $\mathfrak{C}$ by finitely many disks $D = D( \bar{\mathbf{E}}, \xi)
= \{ \mathbf{E} \,|\, | \mathbf{E} - \bar{\mathbf{E}} | \leq \xi \}$
for $\xi < e/2$. It is sufficient to show that for any $t > 0$, there
exists $\eta > 0$ independent of $\bar{\mathbf{E}} \in \mathfrak{C}$, such that for any $\mathbf{E} \in \mathfrak{C}$, $P^{t}(
\mathbf{E}, \cdot) \geq \eta U_{D}(\cdot)$ for all $D$ in this cover. 

Let $\mathbf{E}$ and $D$ be fixed. We prescribe the following
sequence of events.
\begin{itemize}
  \item[(i)]  On the time intervals $[\frac{(i-1)t}{2N}, \frac{it}{2N} )$ for
each $i = 1, \cdots, N - 1$, site $i$ exchanges energy with site $i+1$. After
the energy exchange, the remaining energy at site $i$ is between $e/2$
and $e$. Other clocks do not ring during this time period. 
\item[(ii)] On the time interval $[\frac{(N - 1)t}{2N}, \frac{1}{2 })$,
  site $N$ exchanges energy with the right heat bath. After the energy
  exchange, $E_{N}$ is greater than $\sup_{\mathbf{\bar{E}} \in D} \sum_{i =
  1}^{N} (\bar{E}_{i} + \xi )$. Other clocks do not ring during this time
period. 
\item[(iii)] On the time intervals $[\frac{(N + i -1 )t}{2N}, \frac{(N +
    i)t}{2N} )$ for each $i = 1, \cdots, N$, site $N - i$ exchanges
  energy with site $N + 1 - i$. (Note that site $0$ is the left heat bath.) After each energy exchange, the energy
  left at site $N + 1 - i$ is uniformly distributed in $E_{N+1 - i}
  \in [ \bar{E}_{N+1 - i} - \xi, \bar{E}_{N+1-i} + \xi]$. Other clocks
  do not ring during this time period. 
\end{itemize}

It is then easy to check that the event above occurs with probability
at least $\eta$, where $\eta > 0$ is independent of $\mathbf{E}$
provided $\mathbf{E} \in \mathfrak{C}$.  
\end{proof}

Since $P^{t}(\mathbf{E}, \cdot) $ now has positive density everywhere
in $\mathbb{R}^{N}_{+}$, the strong aperiodicity and irreducibility of
$\mathbf{E}_{n}$ follows immediately.

\medskip

\begin{cor}
\label{aperiod}
$\mathbf{E}_{n}$ is a strongly aperiodic Markov chain.
\end{cor}
\begin{proof}
By theorem \ref{urs}, $\mathfrak{C}$ is a uniform reference
set. In addition $U_{\mathfrak{C}}(\mathfrak{C}) > 0$. The strong aperiodicity follows from
its definition.
\end{proof}

Therefore $\mathbf{E}_{n}$ is aperiodic. 

\medskip

\begin{cor}
\label{irreducible}
$\mathbf{E}_{n}$ is $\lambda$-irreducible, where $\lambda$ is the
Lebesgue measure on $\mathbb{R}^{N}_{+}$. 
\end{cor}
\begin{proof}
Let $A \subset \mathbb{R}^{N}_{+}$ be a set with strictly positive
Lebesgue measure. Then there exists a set $O$ that has the form
$$
  O = \{ \mathbf{E} \,|\, 0 < c \leq E_{i} \leq C < \infty, i = 1, \cdots, N \}
$$
such that $\lambda(O \cap A) > 0$. 

For any $\mathbf{E}_{0} \in \mathbb{R}^{N}_{+}$ and the time
step $h$ as in Theorem 4.1, the same construction as in Theorem \ref{urs} implies that $P^{h}(\mathbf{E}_{0},
\cdot) > \eta U_{O}( \cdot)$ for some $\eta > 0$. Hence $P^{h}( \mathbf{E}_{0}, A) > \eta U_{O}(A)  > 0$.

\end{proof}

\subsection{Absolute continuity of the invariant measure}

This subsection aims to prove the absolute continuity of the invariant
probability measure with respect to the Lebesgue measure. For the sake
of simplicity, we denote the Lebesgue measure on $\mathbb{R}^{N}_{+}$
by $\lambda$. 

\begin{pro}
\label{abscont}
If $\pi$ is an invariant measure of $\mathbf{E}_{t}$, then $\pi$ is
absolutely continuous with respect to $\lambda$ with a strictly positive density. 
\end{pro}

The proof is similar to that of Proposition 6.1 of
\cite{li2014nonequilibrium}. For
$\mathbf{E} \in \mathbb{R}^{N}_{+}$ and $t > 0$, we have decomposition 
$$
  P^{t}( \mathbf{E}, \cdot) = \nu_{\perp} + \nu_{abs} \,,
$$
where $\nu_{abs}$ and $\nu_{\perp}$ are absolutely continuous and
singular component with respect to $\lambda$,
respectively. We need to show that an absolutely continuous component
cannot revert back to singularity as time evolves. 

\begin{lem}
\label{abscont2}
For any probability measure $\mu \ll \lambda$, $\mu P^{t} \ll \lambda$
for any $t  >0$.  
\end{lem}
\begin{proof}
This proof is similar to that of Lemma 6.3 of
\cite{li2014nonequilibrium}. We include it here for the sake of completeness of this paper.  

Let $t > 0$ be fixed. We define $l = (c_{1}, \cdots, c_{n} )$ be the
sequence of energy exchanges, where $c_{i} = k$ means site $k$
exchanges energy with site $k+1$. (As before, heat baths are sites $0$
and $N+1$.) Let $S(l)$ be the event that energy exchanges $(c_{1},
\cdots, c_{n})$ occur during the time period $[0, t)$ in the order
specified, and no other energy exchanges occur. If zero
(resp. infinite many) energy exchange occurs on $[0, t)$, we denote
the corresponding event by $S(\emptyset)$
(resp. $S(\infty)$). Obviously $\mathbb{P}_{\mathbf{E}}[ S(\infty)] =
0$. 

For $S = S(l)$ or $S(\emptyset)$, we define the conditional Markov
operator
$$
  (\mu P_{S})(A) := \int_{\mathbb{R}^{N}_{+}} \mathbb{P}[
  \mathbf{E}_{t} \in A \,|\, \mathbf{E}_{0} = \mathbf{E}  \,|\, S]
  \mu( \mathrm{d} \mathbf{E})
$$
and the measure
$$
  \frac{\mathrm{d} \mu_{Q}}{\mathrm{d} \lambda} (\mathbf{E}):=
  \mathbb{P}_{\mathbf{E}}[S(Q)] \frac{\mathrm{d} \mu}{\mathrm{d}
    \lambda}(\mathbf{E}) 
$$
for $Q = l$ or $\emptyset$. Then by the law of total probability,
$$
  \mu P^{t} = \mu_{\emptyset} + \sum_{l} 
\mu_{l} P_{S(l)} \,.
$$
Therefore, it is sufficient to show that each term above is absolutely
continuous with respect to $\lambda$. Since for each $l = (c_{1},
\cdots, c_{n})$ we have the decomposition
$$
  P_{S(l)} = P_{S(c_{n})} \cdots P_{S(c_{1})} \,.
$$
Hence the proof is reduced to proving the absolute continuity of $\mu
P_{c_{i}}$ for each $i = 0 \sim N+1$, which is a straightforward
exercise. Let $\xi$ and 
$\hat{\xi}_{k}$ be the density of $\mu$ and $\mu P_{S(c_{k})}$,
respectively. Then if $0 < k < N$, we have
\begin{eqnarray*}
 & &\hat{\xi}_{k}( E_{1}, E_{2}, \cdots, E_{N})  \\
&=& \int_{0}^{1}  \xi(
  E_{1}, \cdots, E_{k-1},  p(E_{k} + E_{k+1}), (1 - p)(E_{k } +
  E_{k+1}), E_{k+2}, \cdots, E_{N}) \mathrm{d} p \,.
\end{eqnarray*}
For $k = 0$ and $N$, we have
\begin{eqnarray*}
 & &\hat{\xi}_{0}( E_{1}, E_{2}, \cdots, E_{N}) \\
&=& \int_{0}^{\infty} \int_{(E_{1} - \hat{E}_{1})^{+}}^{\infty}
\xi(\hat{E}_{1}, E_{2}, \cdots, E_{N}) \frac{1}{\hat{E}_{1} + E}
\frac{1}{T_{L}}e^{-E/T_{L}} \mathrm{d} E \mathrm{d} \hat{E}_{1}
\end{eqnarray*}
and
\begin{eqnarray*}
 & &\hat{\xi}_{N}( E_{1}, E_{2}, \cdots, E_{N}) \\
&=& \int_{0}^{\infty} \int_{(E_{N} - \hat{E}_{N})^{+}}^{\infty}
\xi(E_{1}, E_{2}, \cdots, \hat{E}_{N}) \frac{1}{\hat{E}_{N} + E}
\frac{1}{T_{R}}e^{-E/T_{R}} \mathrm{d} E \mathrm{d} \hat{E}_{N} \,.
\end{eqnarray*}

\end{proof}

\begin{proof}[{\bf Proof of Proposition \ref{abscont}. }] Let $\pi =
  \pi_{abs} + \pi_{\perp}$ be
  an invariant measure. Assume $\pi_{\perp} \neq 0$. For $t > 0$,
  $\pi_{abs}P^{t} \ll \lambda$ by Lemma \ref{abscont2}. By Theorem
  \ref{urs}, for any $\mathbf{E} \in \mathfrak{C}$, 
$P^{t/2}( \mathbf{E}, \cdot)$ has a nonzero absolutely continuous
component with respect to the Lebesgue measure, which has a strictly positive density on
$\mathfrak{C}$. Since $\mathfrak{C}$ is accessible within finitely many
energy exchanges, $P^{t/2}( \mathbf{E}, \mathfrak{C} ) > 0$ for all $\mathbf{E}
\in \mathbb{R}^{N}_{+}$. Hence for all $\mathbf{E} \in
\mathbb{R}^{N}_{+}$, $P^{t}( \mathbf{E}, \cdot)$ has a
nonzero absolutely continuous component with respect to the Lebesgue
measure, which has a strictly positive
density on $\mathfrak{C}$.

If $\pi_{\perp} \neq 0$, there must exist $M_{2}, M_{3} < \infty$ such that 
$$
 \pi_{\perp} \left ( \{ \mathbf{E} \,|\, V( \mathbf{E}) \leq M_{2}, W( \mathbf{E})
  \leq M_{3} \} \right ) > 0 \,.
$$
Therefore, Theorem \ref{urs} implies that $\pi_{\perp}P^{t}$ must have an absolutely continuous component. The
absolutely continuous component of $\pi P^{t}$ is strictly larger 
than that of $\pi$. This contradicts to the invariance of $\pi$.  

\end{proof}

\subsection{Excursion time before entering $\mathfrak{C}$}

Now we are ready to estimate the tail of $\tau_{\mathfrak{C}}$. Let
$\hat{V} = \max\{1, V \}$ and $\hat{W} = \max \{ 1, W \}$. 

\begin{thm}
\label{thm72}
For any $\epsilon > 0$, there exists a constant $C_{6} < \infty$ such
that
$$
  \mathbb{P}_{\mathbf{E}}[ \tau_{\mathfrak{C}} > n] \leq C_{6}(
  \hat{W}(\mathbf{E}) + \hat{V}( \mathbf{E}) ) n^{ -(\hat{\alpha} - \epsilon)}
$$
for any $\mathbf{E} \in \mathbb{R}^{N}_{+}$, where $\hat{\alpha} =
2 - 2 \eta$. 
\end{thm}
\begin{proof}
Note that $\mathfrak{C} \subset G$, therefore one can define $\hat{\tau}_{\mathfrak{C}}$ as
the first passage time to $\mathfrak{C}$ for the induced chain $\hat{\mathbf{E}}_{n}$. 

Note that $M_{0}, M_{1} > 1$. Apply Theorem \ref{expfirstpassage} to $\hat{W} = \max\{1, W(
\mathbf{E}) \}$. It follows from Theorem \ref{expfirstpassage} and
Theorem \ref{lyapunovinduced} that there exist constants
$C_{7} > 0$ and $r > 1$ such that, for every $\mathbf{E} \in G$,
$$
  \mathbb{E}_{\mathbf{E}}[ r^{ \hat{\tau}_{\mathfrak{C}}} ] \leq
  \mathbb{E}_{\mathbf{E}}[ \sum_{k = 0}^{\tau_{\mathfrak{C}} - 1}
  W(\hat{\mathbf{E}}_{k}) r^{k} ]  < \hat{W}(
  \mathbf{E}) C_{7} \,.
$$

Applying Markov's inequality to $r^{\hat{\tau}_{C}}$, we have
\begin{equation}
\label{tauhat}
  \mathbb{P}_{\mathbf{E}}[ \hat{\tau}_{\mathfrak{C}} > n] < C_{7} \hat{W}( \mathbf{E}) e^{- c
  n} \,,
\end{equation}
where $c = \log r$. For any given initial condition
$\mathbf{E} \in \mathbb{R}^{N}_{+}$,  we have
\begin{eqnarray*}
 & & \mathbb{P}_{ \mathbf{E}}[ \hat{\tau}_{\mathfrak{C}} > n] \\
&=& \int_{\mathbb{R}^{N}_{+}} \mathbb{P}[ \hat{\tau}_{\mathfrak{C}} > n\,|\,
\mathbf{E}_{T_{1}} = \tilde{E}] \mathbb{P}_{\mathbf{E}}[
\mathbf{E}_{T_{1}} = \mathrm{d}\tilde{E} ]\\
&\leq& \int_{\mathbb{R}^{N}} C_{7} \hat{W}( \tilde{E} )  e^{-c
  n} \mathbb{P}_{ \mathbf{E}}[ \mathbf{E}_{T_{1}} =
\tilde{E}] \\
&=& C_{7} e^{-c n} \mathbb{E}_{\mathbf{E}}[ \hat{W}( \mathbf{E}_{T_{1}})] \\
&\leq& C_{7} e^{-c n} \cdot ( \hat{W}(\mathbf{E}) +
\mathbb{E}_{\mathbf{E}}[ F_{I}([0, T_{1}))]\\
&\leq &  C_{7} e^{-c n}  \cdot( \hat{W}(\mathbf{E}) + \max
\{C_{3}, C_{3}' \hat{V}(\mathbf{E}) \} )  \\
&\leq&  C_{7} e^{-c n}  \cdot( \hat{W}(\mathbf{E}) + \max
\{C_{3}, C_{3}'\} \hat{V}(\mathbf{E})  ) \,,
\end{eqnarray*}
where the third line follows from Equation \eqref{tauhat} and the
second to last line follows from Lemma \ref{inflow3}, constants $C_{3}$ and
$C'_{3}$ are as in Lemma \ref{inflow3}.

By Proposition \ref{lem61}, we have
$$
  \mathbb{E}[ (T_{n+1} - T_{n})^{\hat{\alpha}} \,| \,
  \mathbf{\hat{E}}_{n}] \leq C_{5}\hat{V}( \mathbf{\hat{E}}_{n}) \,,$$
where $C_{5}$ is from Proposition \ref{lem61}. Without loss of
generality, we let $C_{5} \geq 1$. 

Applying Markov's inequality to $(T_{n+1} - T_{n})^{\hat{\alpha}}$,  we have
$$
  \mathbb{P}[ T_{n+1} - T_{n} > k \,|\, \mathbf{E}_{T_{n}} ] \leq C_{5}
  \hat{V}( \mathbf{E}_{T_{n}}) k^{- \hat{\alpha}} \,.
$$
If $n \geq 1$, since $\mathbf{E}_{T_{n}} = \mathbf{\hat{E}}_{n} \in G$, we have a uniform
bound $\hat{V}( \mathbf{\hat{E}}_{n}) \leq M_{0}$.

Therefore, assumptions of  Theorem \ref{inducedchain} are satisfied
for $\xi = \hat{V}$ and $\eta = \hat{W}$. For each given $\epsilon >
0$, notice that $\hat{V} \geq 1$ and $\hat{W} \geq 1$, we have
$$
  \mathbb{P}_{\mathbf{E}}[ \tau_{\mathfrak{C}} > n  ] \leq c ( C_{5}
  \hat{V}( \mathbf{E}) + \hat{W}(\mathbf{E}) + \max
\{C_{3}, C_{3}'\} \hat{V}(\mathbf{E}) )n^{-(\hat{\alpha} - \epsilon)} \leq C_{6}( \hat{W}(
\mathbf{E}) + \hat{V}( \mathbf{E})) n^{-(\hat{\alpha} -
  \epsilon)} \,,
$$
where $C_{6}$ is a constant depending on $\epsilon$ and $\eta$. This completes the proof.

\end{proof}

{\bf Proof of Theorem 1 and 2:} We first prove Theorem 1 and 2 for
$\mathbf{E}_{n}$.

It follows from Theorem \ref{thm72}
that
$$
  \mathbb{P}_{\mathbf{E}}[ \tau_{\mathfrak{C}} > n] \leq C_{6}(
  \hat{W}(\mathbf{E}) + \hat{V}( \mathbf{E}) ) n^{
    -(\hat{\alpha} - \epsilon)} \,,
$$
where $C_{6}$ is a constant that depends on both $\eta$ and
$\epsilon$. Since $\mathbb{P}_{\mathbf{E}}[ \tau_{\mathfrak{C}} > n]$
is at most $1$, we have
$$
  \mathbb{P}_{\mathbf{E}}[ \tau_{\mathfrak{C}} > n] \leq \max\{1, 2 C_{6}\}(
  \hat{W}(\mathbf{E}) + \hat{V}( \mathbf{E}) ) (n + 1)^{
    -(\hat{\alpha} - \epsilon)} \,.
$$
Therefore, by Lemma \ref{t2m}, there exists a constant $C_{8}$ that depends on
$\eta$ and $\epsilon$, such that
$$
  \mathbb{E}_{\mathbf{E}}[ \tau_{\mathfrak{C}}^{\hat{\alpha} - 2 \epsilon}]
  \leq C_{8}( \hat{W}(\mathbf{E}) + \hat{V}( \mathbf{E})) \,.
$$
Note that $\hat{\alpha } = 2 - 2 \eta$. Let $\epsilon = \eta$ and
$\eta = \frac{1}{4} \gamma$, we have 
\begin{equation}
\label{finitemoment}
  \mathbb{E}_{\mathbf{E}_{0}}[ \tau_{\mathfrak{C}}^{2 - \gamma}]
  \leq C_{8}( \hat{W}(\mathbf{E}) + \hat{V}( \mathbf{E})) \,.
\end{equation}
By the compactness of $\mathfrak{C}$, it is easy to see that
$$
  \sup_{\mathbf{E} \in \mathfrak{C}} \mathbb{E}_{\mathbf{E}}[ \tau^{2 -
    \gamma}_{\mathfrak{C}} ] < \infty \,.
$$
Note that $\hat{W} \leq W + 1$ and $\hat{V} \leq V + 1$. Therefore,
for any probability measure $\mu \in \mathcal{M}_{\gamma}$, $C_{8}(
\hat{W}(\mathbf{E}) + \hat{V}( \mathbf{E}))$ is integrable and 
\begin{equation}
\label{pfthm2}
  \mathbb{E}_{\mu}[\tau^{2 - \gamma}_{\mathfrak{C}}] < \infty \,.
\end{equation}

Theorem 1 for $\mathbf{E}_{n}$ is then proved by applying Theorem \ref{thm31} to $\mu, \nu
\in \mathcal{M}_{\gamma}$.

\medskip 

It follows from Corollary \ref{aperiod} and \ref{irreducible} that
$\mathbf{E}_{n}$ is a strongly aperiodic irreducible Markov chain. The
existence of an invariant measure $\pi$ then follows from 
$$
 \sup_{\mathbf{E} \in \mathfrak{C}} \mathbb{E}_{\mathbf{E}}[ \tau_{\mathfrak{C}} ] < \sup_{\mathbf{E} \in \mathfrak{C}} \mathbb{E}_{\mathbf{E}}[ \tau^{2 -
    \gamma}_{\mathfrak{C}} ] < \infty
$$
and Theorem \ref{exist}. The absolute continuity of $\pi$ comes from Proposition
\ref{abscont}. The $n^{-(1 - \gamma)}$ speed of convergence to $\pi$
is given by the existence of $\pi$ and Theorem \ref{thm31}.  

\medskip

It remains to prove the uniqueness of $\pi$. We prove the uniqueness
of $\pi$ for any $\mathbf{E}_{t}$ instead of $\mathbf{E}_{n}$. 

Recall the proof of Proposition \ref{abscont}, for any $t > 0$ and
$\mathbf{E} \in \mathbb{R}^{N}_{+}$, $P^{t}( \mathbf{E}, \cdot)$ has a strictly positive density on
$\mathfrak{C}$. This implies $\mathfrak{C}$ belongs to the support of
any invariant probability measure. However, any two distinct ergodic
invariant probability measures must be mutually singular. In addition,
every invariant probability measure must be a convex combination of
ergodic invariant measures. Hence $\mathbf{E}_{t}$ has at most one invariant
probability measure, which must be $\pi$. (See for example Theorem 1.7
of \cite{hairer2010convergence}.)

This completes the proof of Theorem 2 for $\mathbf{E}_{n}$.

\medskip

The return time argument for $\mathbf{E}_{n}$ also help us to obtain
the tail of a marginal distribution of $\pi$. Since $\pi$ is
absolutely continuous with respect to the Lebesgue measure, any
marginal distribution of $\pi$ also has absolute continuity. Let $\rho_{i}(E)$ be the
density of the marginal distribution of $\pi$ with respect to site
$i$. Let
$$
  q_{i}(E) = \int_{0}^{E}\rho_{i}(\hat{E}) \mathrm{d} \hat{E} =
  \mathbb{P}_{\pi}[E_{i} < E]
$$
be the marginal distribution function. The following lemma holds. 

\begin{lem}
\label{nesstail}
For any $i = 1, \cdots, N$ and any sufficiently small $\gamma > 0$,
there exists $0 < \delta < 1$ such that 
$$
q_{i}(E) \geq  E^{1/2 + \gamma} 
$$
if $0 < E < \delta$.
\end{lem}
\begin{proof}
Define sets $A = \{(E_{1}, \cdots, E_{N}) \,|\, 1 \leq
E_{i} \leq 2 \mbox{ for all } i \}$ and $B_{i}(E) = \{ (E_{1}, \cdots, E_{N})
\,|\, E_{i} \leq E \}$. It is then well
known that $q_{i}(E) = \pi(B_{i}(E))$ is equal to the
expected occupation time for $\mathbf{E}_{n}$ on $B_{i}(E)$, i.e.,
$$
  \pi(B_{i}(E) ) = \int_{A} \pi(\mathrm{d}y)
  \mathbb{E}_{y}[ \sum_{k = 0}^{\tau_{A} - 1} \mathbf{1}_{\{
    \mathbf{E}_{k} \in B_{i}(E)\}}] \,.
$$
(Theorem 10.4.9 of \cite{meyn2009markov}). Let $h$ be the time step
size when defining $\mathbf{E}_{n}$ in Theorem 4.1. We have
$$
   \pi(B_{i}(E) ) \geq \int_{A} \pi(
   \mathrm{d}y)\mathbb{P}_{y}[S_{1}] \mathbb{P}[S_{2} \,|\, S_{1}]
   \mathbb{P}[S_{3} \,|\, S_{1}, S_{2}] E^{-1/2} \,,
$$
where $S_{1}$, $S_{2}$, $S_{3}$ are the following three events:
$$
  S_{1} = \{\mbox{ clock } i \mbox{ rings exactly once at } t_{0} < h , \mbox{ all
    other clocks are silent }\} \,,
$$
$$
  S_{2} = \{ E_{i}(t_{0}^{+}) \in (0, E ) \mbox{ after } S_{1} \mbox{occurs}\} \,,
$$
and
$$
  S_{3} = \{ \mbox{ after } t = h, \mbox{ no energy exchange involves }
  E_{i} \mbox{ before } t = h + \left \lfloor E^{-1/2} \right \rfloor\} \,.
$$
Then it is easy to see that $P_{y}[S_{1}]$ is uniformly positive for
$y \in A$, $\mathbb{P}[S_{2} \,|\, S_{1}] > E/4$ is independent of the initial
condition, and $\mathbb{P}[ S_{3} \,|\, S_{1}, S_{2} ] \geq e^{-1}$ for any initial
condition because $E_{i}(h) < E$. In addition we have $\pi(A) > 0$. Hence for all sufficiently small $E > 0$, we have
$$
  q_{i}(E)  \geq c E^{1/2}
$$
for some constant $c > 0$ that is independent of $E$ and $i$. This
completes the proof.

\end{proof}

\medskip 

{\bf Theorem 1 and 2} for $\mathbf{E}_{t}$. The last step is to pass results
from $\mathbf{E}_{n}$ to $\mathbf{E}_{t}$. The contraction of the Markov
operator is easy to pass because we have 
$$
  \|\mu P^t - \nu P^t\|_{\rm TV} = 
\| (\mu P^{\lfloor  \frac{t}{h} \rfloor h}  - \nu P^{\lfloor
  \frac{t}{h} \rfloor h}) P^{(t-\lfloor  \frac{t}{h} \rfloor) h}\|_{\rm TV}  
\le \|\mu P^{\lfloor  \frac{t}{h} \rfloor h} - \nu P^{\lfloor
  \frac{t}{h} \rfloor h} \|_{\rm TV} \ \,.
$$

It remains to show that $\pi$, the invariant probability measure of $\mathbf{E}_{n}$, is
invariant for any $\mathbf{E}_{t}$, $t > 0$.

\begin{lem}
$\pi P^{t} = \pi$ for any $t > 0$. 
\end{lem}
\begin{proof}
Note that the argument in Section 4 and 5 works for all sufficiently
small time steps. Let $h$ be the time step we have chosen for
$\mathbf{E}_{n}$. For any $r < h$, $P^{r}$ also admits an invariant
measure $\pi_{r}$. It is then sufficient to show that $\pi_{r} = \pi$
because any $t$ can be written as 
$$
  t = \lfloor  \frac{t}{h} \rfloor \cdot h + r
$$
for some $r < h$. 

In addition, we have the ``continuity at zero''. 
$$
  \| \pi_{r} P^{\delta} - \pi_{r} \|_{TV} \rightarrow 0 \quad \mbox{ as }
  \delta \rightarrow 0
$$
because all clock rates are less than $K$. 

Without loss of generality, assume $r/h \notin
\mathbb{Q}$. By the density of orbits in irrational rotations, there exist
sequences $i_{n}$, $j_{n} \in \mathbb{Z}^{+}$, such that $\delta_{n}
:= h - \frac{i_{n}}{j_{n}} r  
\rightarrow 0$ from right. Then 
$$
  \pi_{r}P^{h} = \pi_{r}P^{\frac{i_{n}}{j_{n}} r } P^{\delta_{n}} =
  \pi_{r} P^{\delta_{n}} \rightarrow \pi_{r} \,
$$
by the ``continuity at zero''.

Hence $\pi_{r}$ is invariant with respect to $P^{h}$. By uniqueness,
$\pi_{r} = \pi$. 

\end{proof}

Therefore, Theorem 1 and 2 also hold for $\mathbf{E}_{t}$.

\medskip

{\bf Remark: } 
We expect Theorem 1 to be close to optimal. Let
$\mathfrak{C}$ be a uniform reference set. Then it is easy to see that
there exists an $\epsilon > 0$ such that $B_{\epsilon} = \{ \mathbf{E}
\,|\, E_{i} < \epsilon \mbox{ for some } i = 1, \cdots, N \}$ is disjoint
with $\mathfrak{C}$. Let $D = \{ \mathbf{E} \,|\, E_{i} < L \}$ for
some large $L$. Then it is easy to see that for any $\mu$ and $\nu$ that
have uniformly positive density on $B_{\epsilon} \cap D$, we have $\mu (
B_{t^{-2}}) \sim O(t^{-2}) $ and $\nu ( B_{t^{-2}}) \sim O(t^{-2}) $
for $t \gg 1$. When $E_{i} < t^{-2}$,
the probability that no energy exchange occurs between site $i - 1$
and $i$, or between site $i$ and $i+1$, before time $t$ is
$O(1)$. Therefore, we have the lower bound on tails $\mathbb{P}_{\mu}[
\tau_{\mathfrak{C}} > t] \geq O(1) \cdot t^{-2}$ (and $\mathbb{P}_{\nu}[
\tau_{\mathfrak{C}} > t] \geq O(1) \cdot t^{-2}$). This implies the
coupling time $T$ has the tail
$$
  \mathbb{P}_{\mu, \nu }[ T > t] \geq O(1) \cdot t^{-2} \,.
$$
This is consistent with our numerical result in
\cite{li2017numerical} that the tail of
$\mathbb{P}[\tau_{\mathfrak{C}} > t]$ is $\sim t^{-2}$. Similar argument leads to the proof of {\it Proposition 4}.

\medskip

{\bf Proof of Theorem 3.} 
The following calculation is straightforward.
\begin{eqnarray*}
& &  \left|\int (P^{t} \zeta)(
  \mathbf{E}) \xi( \mathbf{E}) \mu( \mathrm{d}\mathbf{E}) - 
  \int (P^{t}\zeta)( \mathbf{E}) \mu(\mathrm{d}\mathbf{E}) \int \xi( \mathbf{E}) \mu(
  \mathrm{d} \mathbf{E}) \right| \\
& = & \left| \int \xi (\mathbf{E}) \left( (P^{t}\zeta)(\mathbf{E}) - 
\int (P^{t}\zeta)(\mathbf{Z}) \mu(\mathrm{d}\mathbf{Z}) \right) 
\mu( \mathrm{d} \mathbf{E}) \right|\\
& \le & \|\xi\|_{L^{\infty}} \ \|\zeta\|_{L^{\infty}} \ \int \|\delta_{\mathbf{E}}P^t - \mu P^t\|_{TV} \ 
\mu( \mathrm{d} \mathbf{E}) \,.
\end{eqnarray*}
It then follows from Corollary \ref{cor35} and equation
\eqref{finitemoment} that 
$$
   \|\delta_{\mathbf{E}}P^t - \mu P^t\|_{TV} \leq C (\hat{W}(
   \mathbf{E}) + \hat{V}( \mathbf{E}) + C_{\mu}) ( \left \lfloor t \right
   \rfloor )^{\gamma - 2} 
$$
for some $C, C_{\mu}< \infty$ that is independent of
$\mathbf{E}$. Since 
$\hat{W}(\mathbf{E}) + \hat{V}( \mathbf{E})$ is $\mu$-integrable, we
have
$$
  \left|\int (P^{t} \zeta)(
  \mathbf{E}) \xi( \mathbf{E}) \mu( \mathrm{d}\mathbf{E}) - 
  \int (P^{t}\zeta)( \mathbf{E}) \mu(\mathrm{d}\mathbf{E}) \int \xi( \mathbf{E}) \mu(
  \mathrm{d} \mathbf{E}) \right| \leq O(1) \cdot \|\xi\|_{L^{\infty}}
\ \|\zeta\|_{L^{\infty}} \ t^{\gamma - 2} \,,
$$
where the $O(1)$ term depends on $\gamma, N$, and $\mu$.  \qed

{\bf Proof of Proposition 4.}

Let $\nu$ be a probability measure that satisfies the following
properties.
\begin{itemize}
\item $\nu$ is absolutely continuous with respect to $\pi$.
\item Let $B_{\epsilon} = \{ \mathbf{E} = (E_{1}, \cdots, E_{N}) \in \mathbb{R}^{N}_{+} \,|\, E_{1} <
  \epsilon \}$ for some fixed small $\epsilon > 0$. $\nu$ satisfies
  $\mathrm{d}\nu/ \mathrm{d} \pi = 4$ on $B_{\epsilon}$.
\end{itemize}
Such $\nu$ must exist because $\pi$ is absolutely continuous with
respect the Lebesgue measure. Hence we can always find a small $\epsilon >
0$ such that $\pi(B_{\epsilon}) < \frac{1}{4}$.
\medskip

Then for $t > 0$, we have
\begin{eqnarray*}
\| \nu P^{t} - \pi  \|_{TV}& \geq& \| (\nu P^{t}) (B_{t^{-2}}) -
                                   \pi(B_{t^{-2}}) \|_{TV}  \geq (\nu P^{t})
  (B_{t^{-2}}) - \pi(B_{t^{-2}})\\
&\geq& \mathbb{P}_{\nu}[ \mbox{clock } 1 \mbox{ does
       not ring before } t, \mathbf{E}_{0} \in B_{t^{-2}}] -
       \pi(B_{t^{-2}}) \\
&=& \nu(B_{t^{-2}}) \mathbb{P}_{\nu |_{B_{t^{-2}}}}[ \mbox{clock } 0
    \mbox{ and } 1 \mbox{ does
       not ring before } t ] - \pi(B_{t^{-2}}) \\
&\geq& \nu(B_{t^{-2}})e^{-1} - \pi(B_{t^{-2}}) = (4 e^{-1} - 1)\pi(B_{t^{-2}}) \,,
\end{eqnarray*}
where $\nu |_{B_{t^{-2}}}$ is the restricted probability measure $\nu$
on $B_{t^{-2}}$. Let $\gamma > 0$ be a sufficiently small
number. Apply Lemma \ref{nesstail} to the marginal distribution function $q_{1}(E)$ and
small parameter $\gamma/2$. There should exist a $T < \infty$ such that
$$
  \pi(B_{t^{-2}}) = q_{1}(t^{-2}) \geq t^{-1 - \gamma}
$$
for any $t > T$. Hence
$$
  \|\nu P^{t} - \pi\|_{TV} \geq (4 e^{-1} - 1)(1+t)^{ -1 - \gamma} , \qquad t > T \,.
$$
In addition, we can always find $c > 0$ such that
$$
  \|\nu P^{t} - \pi\|_{TV} \geq c (1 + t)^{-1 - \gamma} 
$$
for all $0 \leq t \leq T$. The proof is completed by combining the two
estimates. \qed

{\bf Remark:}
The result of Lemma \ref{nesstail} implies that the invariant probability measure
$\pi$ may not belong to $\mathcal{M}_{\eta}$ for sufficiently small
$\eta > 0$. As a result, the initial probability distribution $\nu$ constructed in
the proof of {\bf Proposition 4} may not be in $\mathcal{M}_{\eta}$ either. Our numerical
simulation shows that the lower bound of convergence holds for many initial probability
distributions within the measure class $\mathcal{M}_{\eta}$ as well. But a rigorous proof
requires many detailed properties of $\pi$, which turns out to be very difficult due to
the nonequilibrium nature of the system.

\section*{Acknowledgement}

The author would like to thank Lai-Sang Young, Jonathan Mattlingly,
and Martin Hairer for many enlightening discussions.

\bibliography{ref}
\bibliographystyle{amsplain}
\end{document}